\documentclass[11pt]{extarticle}
\usepackage{graphicx}
\usepackage{amsthm}
\usepackage{amsmath}
\usepackage{amsfonts}
\usepackage[numbers]{natbib}
\usepackage{fullpage}
\usepackage{amssymb}
\usepackage{graphics}
\usepackage{mathdots}
\usepackage{enumerate}
\usepackage{hyperref}

\newtheoremstyle{tnp}
  {\topsep}   
  {\topsep}   
  { \itshape}  
  {0pt}       
  {\bfseries} 
  {.}         
  {5pt plus 1pt minus 1pt}  
  {\thmname{#1}\thmnumber{ #2}\if\relax\detokenize{#3}\relax\else
  {\normalfont \; \thmnote{#3}}\fi}          
  \theoremstyle{tnp}

\theoremstyle{tnp}	 				\newtheorem{theorem}{Theorem}[section]
\theoremstyle{plain}	 				\newtheorem{conjecture}{Conjecture}[section]
\theoremstyle{plain}	 				\newtheorem{conjecturesub}{Conjecture}[subsection]
\theoremstyle{plain} 					\newtheorem{corollary}{Corollary}[section]
\theoremstyle{plain} 					\newtheorem*{query}{Query}

\usepackage[usenames]{color}

\DeclareMathOperator\bind{bind}
\DeclareMathOperator\BM{BM}
\DeclareMathOperator\defi{def}
\begin{document}

\title{\bf A Survey of Best Monotone Degree Conditions\\
  for Graph Properties}

\date{}      
\maketitle

\vspace{-40pt}

\noindent
\[\text{D. Bauer}\,^\mathrm{1},\quad \text{H.J. Broersma}\,^\mathrm{2},\quad
\text{J. van den Heuvel}\,^\mathrm{3},\quad \text{N. Kahl}\,^\mathrm{4},\]
\[\text{A. Nevo}\,^\mathrm{1},\quad \text{E. Schmeichel}\,^\mathrm{5},\quad
\text{D.R. Woodall}\,^\mathrm{6},\quad \text{M. Yatauro}\,^\mathrm{7}\]

\bigskip
\noindent {\footnotesize $^{\mathrm{1}}$\,\textit{Department of
    Mathematical Sciences, Stevens Institute of Technology, Hoboken, NJ
    07030,
    USA}}\\[0.5mm]
{\footnotesize $^{\mathrm{2}}$\,\textit{Faculty of EEMCS, University of
    Twente, P.O. Box 217, 7500 AE Enschede, The Netherlands}}\\[0.5mm]
{\footnotesize $^{\mathrm{3}}$\,\textit{Department of Mathematics, London
    School of Economics, Houghton Street, London WC2A 2AE, UK}}\\[0.5mm]
{\footnotesize $^{\mathrm{4}}$\,\textit{Department of Mathematics and
    Computer Science, Seton Hall University, South Orange, NJ 07079,
    USA}}\\[0.5mm]
{\footnotesize $^{\mathrm{5}}$\,\textit{Department of Mathematics, San
    Jos\'{e} State University, San Jos\'{e}, CA 95192, USA}}\\[0.5mm]
{\footnotesize $^{\mathrm{6}}$\,\textit{School of Mathematical Sciences,
    University of Nottingham, Nottingham NG7 2RD, UK}}\\[0.5mm]
{\footnotesize $^{\mathrm{7}}$\,\textit{Department of Mathematics, Penn
    State University, Brandywine Campus, Media, PA 19063, USA}}\\[2mm]
{\footnotesize Email: \url{dbauer@stevens.edu},
  \url{h.j.broersma@utwente.nl}, \url{j.van-den-heuvel@lse.ac.uk},
  \url{kahlnath@shu.edu},\\
  \mbox{}\hphantom{Email: }\url{anevo@stevens.edu},
  \url{schmeichel@math.sjsu.edu}, \url{douglas.woodall@nottingham.ac.uk},
  \url{mry3@psu.edu}}

\bigskip
\begin{abstract}
  \noindent
  We survey sufficient degree conditions, for a variety of graph
  properties, that are best possible in the same sense that Chv\'{a}tal's
  well-known degree condition for hamiltonicity is best possible.
\end{abstract}

\section{Introduction}
\label{sec1}

We consider only finite graphs without loops or multiple edges. Our
terminology and notation are standard except as indicated. A good reference
for undefined terms is~\cite{CLZ11}. We mention only that given graphs
$G,H$ with disjoint vertex sets, we will denote their disjoint union by
$G\cup H$, and their join by $G+H$.

We generally use the standard abbreviation for integer sequences; e.g.,
$(4,4,4,4,4,5,5,6)$ will be denoted $4^55^26^1$. An integer sequence
$\pi=(d_1\le\cdots\le d_n)$ is called \emph{graphical} if there exists a
graph~$G$ having~$\pi$ as its vertex degree sequence; in that case, $G$ is
called a \emph{realization} of~$\pi$. If~$P$ is a graph property, such as
`hamiltonian' or `$k$-connected', we call a graphical sequence~$\pi$
\emph{forcibly~$P$} if every realization of~$\pi$ has property~$P$. If
$\pi=(d_1\le\cdots\le d_n)$ and $\pi'=( d_1'\le\cdots\le d_n')$ are integer
sequences, we say \emph{$\pi'$ majorizes~$\pi$}, denoted $\pi'\ge\pi$, if
$d_i'\ge d_i$ for $1\le i\le n$. There is an analogous definition and
notation for $\pi'$ \emph{minorizes~$\pi$}.

Historically, the vertex degrees of a graph have been used to provide
sufficient conditions for the graph to have certain properties, such as
hamiltonicity or $k$-connectedness. In particular, sufficient conditions
for~$\pi$ to be forcibly hamiltonian were given by several authors in
\cite{B69,D52,P62}, culminating in the following theorem of
Chv\'{a}tal~\cite{C72}.

\begin{theorem}[(Chv\'{a}tal~\cite{C72})]\label{thm:11}\mbox{}\\*
  Let $\pi=(d_1\le\cdots\le d_n)$ be a graphical sequence, with $n\ge3$. If
  \begin{equation}\label{eq:11}
    d_i\le i\; \Rightarrow\; d_{n-i}\ge n-i,\quad
    \text{for $1\le i\le\tfrac12(n-1)$},
  \end{equation}
  then~$\pi$ is forcibly hamiltonian.
\end{theorem}

\noindent
Unlike its predecessors, Theorem~\ref{thm:11} has the property that if a
sequence~$\pi$ fails to satisfy condition~\eqref{eq:11} for some index~$i$,
then~$\pi$ is majorized by the sequence $\pi'=i^i(n-i-1)^{n-2i}(n-1)^i$,
with nonhamiltonian realization $G'=K_i+(\overline{K_i}\cup K_{n-2i})$. As
we will see in Section~\ref{sec2}, this key property implies that
condition~\eqref{eq:11} in Theorem~\ref{thm:11} is the best of an entire
important class of degree conditions for~$\pi$ to be forcibly hamiltonian.

Sufficient conditions for $\pi$ to be forcibly $k$-connected have been
given by several authors in \cite{CH68,CKK68}, culminating in the following
theorem of Bondy~\cite{B69} (though the form in which we present it is due
to Boesch~\cite{B74}).

\begin{theorem}[(Bondy~\cite{B69})]\label{thm:12}\mbox{}\\*
  Let $\pi=(d_1\le\cdots\le d_n)$ be a graphical sequence, with $n\ge2$ and
  $1\le k\le n-1$. If
  \begin{equation}\label{eq:12}
    d_i\le i+k-2\; \Rightarrow\; d_{n-k+1}\ge n-i,\quad
    \text{for $1\le i\le\tfrac12(n-k+1)$},
  \end{equation}
  then~$\pi$ is forcibly $k$-connected.
\end{theorem}

\noindent
If the sequence~$\pi$ fails to satisfy~\eqref{eq:12} for some index~$i$,
then~$\pi$ is majorized by $\pi'=(i+k-2)^i\linebreak[1]
(n-i-1)^{n-k-i+1}(n-1)^{k-1}$, with not-$k$-connected realization
$G'=K_{k-1}+(K_i\cup K_{n-k-i+1})$. Thus~\eqref{eq:12} is the best
condition for~$\pi$ to be forcibly $k$-connected in precisely the same
way~\eqref{eq:11} is the best condition for~$\pi$ to be forcibly
hamiltonian.

Our goal in this paper is to survey degree conditions, for a wide variety
of graph properties, that are best in precisely the same sense as
conditions~\eqref{eq:11} and~\eqref{eq:12} above. In Section~\ref{sec2}, we
present a formal framework which allows us -- at least in principle -- to
identify and construct such degree conditions, and to evaluate their
inherent complexity. In Section~\ref{sec3}, we apply this framework, and
consider best degree conditions for the following graph properties and
parameters: edge-connectivity (Section~\ref{sec3.1}), binding number
(Section~\ref{sec3.2}), toughness (Section~\ref{sec3.3}), existence of
factors (Section~\ref{sec3.4}), existence of paths and cycles
(Section~\ref{sec3.5}), and independence number, clique number, chromatic
number, and vertex arboricity (Section~\ref{sec3.6}). In
Section~\ref{sec4}, we consider best degree conditions for~$\pi$ to be
forcibly $P_1\Rightarrow P_2$. Such conditions represent the least amount
of degree strength that needs to be added to~$P_1$ to get a sufficient
condition for~$P_2$, and are especially interesting when~$P_1$ is a
necessary condition for~$P_2$. We will focus in Section~\ref{sec4} on the
situation where~$P_2$ is `hamiltonian', and~$P_1$ belongs to the set
\{\,`traceable', `$2$-connected', `$1$-binding', `contains a $2$-factor',
`$1$-tough'\,\} of prominent necessary conditions for hamiltonicity. In
Section~\ref{sec5}, we consider situations where~$P_1$ does not
imply~$P_2$, but the best degree condition for~$P_1$ implies the best
degree condition for~$P_2$. Such results may be considered improvements in
a degree sense over what is true structurally.

\setcounter{equation}{0}
\section{Framework for Best Monotone Degree Conditions}
\label{sec2}

A graph property $P$ is \emph{increasing} (\emph{decreasing}) if whenever a
graph~$G$ has~$P$, so does every edge-augmented supergraph (edge-deleted
subgraph) of~$G$. Thus `hamiltonian' and `$k$-connected' are increasing
properties, while `$k$-colorable' is a decreasing property. In the rest of
this section, we assume~$P$ is an increasing graph property; a completely
analogous development can be given for decreasing graph properties.

Given a graph property~$P$, consider a theorem~$T$ which declares certain
graphical sequences forcibly~$P$, rendering no decision on the remaining
graphical sequences. Such a theorem~$T$ is called a (\emph{forcibly})
$P$-theorem. Thus Theorem~\ref{thm:11} is a hamiltonian theorem.

A $P$-theorem~$T$ is \emph{monotone} if whenever~$T$ declares a graphical
sequence~$\pi$ forcibly~$P$, $T$ declares every graphical $\pi'\ge\pi$
forcibly~$P$.

A $P$-theorem~$T$ is \emph{$P$-optimal} (or \emph{optimal}, if~$P$ is
understood) if every graphical sequence which~$T$ does not declare
forcibly~$P$ is not forcibly~$P$. Note that Theorem~\ref{thm:11} is not
optimal in this sense; e.g., Theorem~\ref{thm:11} does not declare
$\pi=(2k)^{4k+1}$ forcibly hamiltonian for any $k\ge1$, but all such~$\pi$
are forcibly hamiltonian~\cite{N71}.

A $P$-theorem~$T$ is \emph{$P$-weakly-optimal} (or \emph{weakly optimal},
if~$P$ is understood) if every graphical sequence which~$T$ does not
declare forcibly~$P$ is majorized by a graphical sequence that is not
forcibly~$P$. As noted in the previous section, Theorem~\ref{thm:11} is
weakly optimal in this sense.

A monotone $P$-theorem that is also $P$-weakly-optimal is a `best' monotone
$P$-theorem in the following sense.

\begin{theorem}\label{thm:21}\mbox{}\\*
  Let $T, T_0$ be monotone $P$-theorems, and let~$T_0$ be
  $P$-weakly-optimal. Then any graphical sequence declared forcibly~$P$
  by~$T$ is also declared forcibly~$P$ by~$T_0$.
\end{theorem}

\begin{proof} Suppose to the contrary that~$T$ declares some graphical
  sequence~$\pi$ forcibly~$P$, but~$T_0$ does not. Since~$T_0$ is
  $P$-weakly-optimal, there exists a graphical $\pi'\ge\pi$ which is not
  forcibly~$P$. But since~$T$ is monotone and declares~$\pi$ forcibly~$P$,
  $T$ must declare $\pi'\ge\pi$ forcibly~$P$, a contradiction.
\end{proof}

\noindent
Since Theorem~\ref{thm:11} is monotone and weakly optimal,
Theorem~\ref{thm:11} is a best monotone hamiltonian theorem. Similarly,
Theorem~\ref{thm:12} is a best monotone $k$-connected theorem.

By Theorem~\ref{thm:21}, all best monotone $P$-theorems declare the same
set of graphical sequences forcibly~$P$; we denote this set of graphical
sequences by $\BM(P)$. So in terms of their effect, all best monotone
$P$-theorems are equivalent.

In the following three paragraphs, we describe a generic way to construct
-- at least in principle -- best monotone $P$-theorems. Consider the
partially-ordered set~$G_n$ whose elements are the graphical sequences of
length~$n$, and whose partial-order relation is degree majorization. The
graphical sequences of length~$n$ that are not forcibly~$P$ induce a
subposet of~$G_n$, denoted~$\overline{P_n}$. A maximal element
in~$\overline{P_n}$ is called a \emph{$(P,n)$-sink}. The set of all
$(P,n)$-sinks will be denoted $S(P,n)$.

Given a graphical sequence $\pi=(a_1\le\cdots\le a_n)$, note that~$\pi$
fails to satisfy the degree condition~$C(\pi)$ defined by
\[C(\pi):\quad d_1\ge a_1+1\;\vee\;\cdots\;\vee\;d_n\ge a_n +1;\]
indeed, $C(\pi)$ is the weakest monotone degree condition which
`blocks'~$\pi$ (i.e., which~$\pi$ fails to satisfy). We call~$C(\pi)$ the
\emph{Chv\'{a}tal-type condition} for $\pi$. In the sequel, we will usually
write~$C(\pi)$ in the more traditional form
\[d_1\le a_1\;\wedge\;\cdots\;\wedge\;d_{j-1}\le a_{j-1}\;\Rightarrow\;
d_j\ge a_j+1\;\vee\;\cdots\;\vee\;d_n\ge a_n +1,\]
for some $j<n$.

If $\pi\in\overline{P_n}$, then by definition there exists $\pi'\in S(P,n)$
majorizing~$\pi$, and thus~$\pi$ fails to satisfy~$C(\pi')$. Put
differently, if a graphical $n$-sequence $\pi$ satisfies the degree
condition $\bigwedge_{\pi\in S(P,n)}C(\pi)$, then~$\pi$ is forcibly~$P$;
i.e., the theorem~$T$ with degree condition
$\bigwedge_{\pi\in S(P,n)}C(\pi)$ is a forcibly $P$-theorem. But
certainly~$T$ is monotone, and~$T$ is also $P$-weakly-optimal (if~$\pi$
fails to satisfy the degree condition of~$T$, then~$\pi$ is majorized by
some $\pi'\in S(P,n)\subseteq\overline{P_n}$ which is not forcibly~$P$).
Thus~$T$ is a best monotone $P$-theorem.

In practical terms, it may be almost impossible to identify the precise set
of sinks $S(P,n)$. Fortunately, it is not necessary to make this precise
identification to get a best monotone \mbox{$P$-theorem}: If one can merely
identify a set~$A$ of non-$P$ graphs on~$n$ vertices whose set of degree
sequences $\prod(A)$ contains all of $S(P,n)$ (so
$S(P,n)\subseteq\prod(A)\subseteq\overline{P_n}$), then -- as above -- the
theorem with degree condition $\bigwedge_{\pi\in\prod(A)}C(\pi)$ will also
be a best monotone $P$-theorem. Although it is typically difficult to find
even such a set~$A$, we will see in the following sections that this is
possible for a remarkable number of graph properties.

Finally, we note that $|S(P,n)|$ may be considered the `inherent
complexity' of a best monotone theorem on~$n$ vertices. More precisely, we
have the following.

\begin{theorem}\label{thm:22}\mbox{}\\*
  When the degree condition of a best monotone $P$-theorem on~$n$ vertices
  is expressed as a conjunction $\bigwedge C(\pi)$ of $P$-weakly-optimal
  Chv\'{a}tal-type conditions, the conjunction must contain at least
  $|S(P,n)|$ such conditions.
\end{theorem}

\begin{proof} It suffices to show that any $\pi \in S(P,n)$ satisfies
  $\bigwedge_{\pi'\in S(P,n)-\{\pi\}}C(\pi')$; for then the conjunction
  must contain each Chv\'{a}tal-type condition~$C(\pi)$, as~$\pi$ ranges
  over $S(P,n)$. Suppose to the contrary that some sink
  $\pi_a=(a_1\le\cdots\le a_n)$ violates $C(\pi_b)$, where
  $\pi_b=(b_1\le\cdots\le b_n)$ is another sink. Then $a_i\le b_i$, for
  $1\le i\le n$, and so $\pi_a\le\pi_b$, contradicting the assumption
  that~$\pi_a$ is a sink.
\end{proof}

\setcounter{section}{2}
\setcounter{equation}{0}
\renewcommand{\thetheorem}{\arabic{section}.\arabic{subsection}.\arabic{theorem}}
\section{Best Monotone Conditions for Graph Properties \boldmath$P$}
\label{sec3}

\subsection{Edge-Connectivity}\label{sec3.1}

We noted in Section~\ref{sec1} that Bondy~\cite{B69} (see also
Boesch~\cite{B74}) gave a best monotone condition for
$k$-vertex-connectedness (Theorem~\ref{thm:12}). While Theorem~\ref{thm:12}
is also a sufficient condition for $k$-edge-connectedness, it is not a best
monotone condition when $k\ge2$.

A best monotone condition for $2$-edge-connectedness was given
in~\cite{BHKS09}.

\begin{theorem}[(Bauer et al.\ \cite{BHKS09})]\label{thm:311}\mbox{}\\*
  Let $\pi = (d_1\le\cdots\le d_n)$ be a graphical sequence. If
  \begin{subequations}
    \begin{gather}
      d_1\ge2;\label{eq:311}\\
      d_i-1\le i-1\;\wedge\;d_i\le i\; \Rightarrow\; d_{n-1}\ge n-i\;
      \vee\; d_n\ge n-i+1,\quad
      \text{for $3\le i<\tfrac12n$};\label{eq:312}\\
      d_{n/2}\le\tfrac12n-1\; \Rightarrow\; d_{n-2}\ge\tfrac12n\; \vee\;
      d_n\ge\tfrac12n+1,\quad \text{if~$n$ is even},\label{eq:313}
    \end{gather}
  \end{subequations}
  then~$\pi$ is forcibly $2$-edge-connected.
\end{theorem}

\noindent
For the weak optimality of Theorem~\ref{thm:311}, let $G(n,i)$, $i\ge1$,
denote disjoint cliques~$K_i$ and $K_{n-i}$ joined by a single edge.
If~$\pi$ fails to satisfy~\eqref{eq:311}, then~$\pi$ is majorized by the
degrees of $G(n,1)$. If~$\pi$ fails to satisfy~\eqref{eq:312} for some~$i$,
then~$\pi$ is majorized by the degrees of $G(n,i)$. If~$\pi$ fails to
satisfy~\eqref{eq:313}, then~$\pi$ is majorized by the degrees of
$G(n,n/2)$. Since none of these graphs is $2$-edge-connected,
Theorem~\ref{thm:311} is weakly optimal.

Kriesell~\cite{K07} and Yin and Guo~\cite{HGTA} independently established a
best monotone condition for $3$-edge-connectedness, which had been
conjectured in \cite{BHKS09}.

\begin{theorem}[(Kriesell~\cite{K07}, Yin \& Guo~\cite{HGTA})]
  \label{thm:312}\mbox{}\\*
  Let $\pi=(d_1\le\cdots\le d_n)$ be a graphical sequence. If
  \begin{subequations}
    \begin{gather}
      d_1\ge3;\label{eq:314}\\
      d_{i-2}\le i-1\;\wedge\;d_i\le i\; \Rightarrow\; d_{n-2}\ge n-i\;
      \vee\; d_n\ge n-i+1,\quad \text{for $4\le i<\tfrac12n$};
      \label{eq:315}\\
      d_{i-1}\le i-1\;\wedge\;d_i\le i+1\; \Rightarrow\;
      d_{n-2}\ge n-i\;\vee\; d_n\ge n-i+1,\quad
      \text{for $4\le i<\tfrac12(n-1)$};\label{eq:316}\\
      d_{i-2}\le i-1\;\wedge\;d_i\le i\; \Rightarrow\; d_{n-1}\ge n-i\;
      \vee\; d_n\ge n-i+2,\quad \text{for $4\le i<\tfrac12n$};
      \label{eq:317}\\
      d_{n/2}\le\tfrac12n-1\; \Rightarrow\; d_{n-4}\ge\tfrac12n\; \vee\;
      d_n\ge\tfrac12n+1,\quad \text{if $n$ is even};\label{eq:318}\\
      d_{(n-3)/2}\le\tfrac12(n-3)\; \Rightarrow\; d_{n-3}\ge\tfrac12(n+1)\;
      \vee\; d_n\ge\tfrac12(n+3),\quad \text{if $n$ is odd};
      \label{eq:319}\\
      d_{n/2}\le\tfrac12n-1\; \Rightarrow\; d_{n-3}\ge\tfrac12n\; \vee\;
      d_{n-1}\ge\tfrac12n+1\; \vee\; d_n\ge\tfrac12n+2,\quad
      \text{if $n$ is even},\label{eq:3110}
    \end{gather}
  \end{subequations}
  then~$\pi$ is forcibly $3$-edge-connected.
\end{theorem}

\noindent
The increase in the number of conditions in Theorem~\ref{thm:312} when
$k=3$, compared to Theorem~\ref{thm:311} when $k=2$, is notable. Indeed, we
now prove that the number of weakly optimal Chv\'{a}tal-type conditions in
a best monotone condition for $k$-edge-connectedness grows
superpolynomially in~$k$, for~$n$ sufficiently large. A more involved proof
of this was given previously by Kriesell~\cite{K07}.

By Theorem~\ref{thm:22}, it suffices to prove the following.

\begin{theorem}\label{thm:313}\mbox{}\\*
  Let $k\ge2$, and let $n\ge4k-2$ be an even integer. Then there are at
  least $p(k-1)$ $k$-edge-connected sinks in~$G_n$, where~$p$ denotes the
  integer partition function, so that
  $p(r)\sim\dfrac{1}{4\sqrt{3}r}e^{\pi\sqrt{\frac{2r}{3}}}$~\cite{HR18}.
\end{theorem}

\begin{proof}[Proof of Theorem~\ref{thm:313}] Construct a family of
  $p(k-1)$ edge-maximal not $k$-edge-connected graphs on~$n$ vertices as
  follows: Begin with disjoint copies $X,Y$ of $K_{n/2}$. Let
  $a_1+a_2+\cdots+a_j$ be any partition of $k-1$, and choose vertices
  $x_1,x_2,\ldots,x_j\in X$. Add $k-1$ edges between~$X$ and~$Y$ so
  that~$a_i$ of these edges are incident at $x_i\in X$, for $1\le i\le j$,
  and the edges are incident to $k-1$ distinct vertices in~$Y$. Call the
  resulting graph $G(a_1,\ldots,a_j)$, noting that it has minimum degree
  $\delta(G(a_1,\ldots, a_j))=\tfrac12n-1$.

  To complete the proof, it suffices to show

  \medskip
  \noindent
  \textbf{Claim}. $\pi(G(a_1,\ldots,a_j))$ is a $k$-edge-connected sink
  in~$G_n$.

  \medskip
  \noindent
  \emph{Proof of the Claim.} Let $G=G(a_1,\ldots,a_j)$. Suppose to the
  contrary that~$\pi(G)$ is majorized by $\pi(H)\neq\pi(G)$, where~$H$ is
  an edge-maximal not $k$-edge-connected graph, necessarily consisting of
  two disjoint cliques $X,Y$ such that $|X|+|Y|=n$ and $|E(X,Y)|= k-1$. We
  may assume $|X|>|Y|$, so that $|Y|<\tfrac12n$ (if $|X|=|Y|$,
  then~$\pi(G)$ and~$\pi(H)$ have the same degree sum and~$\pi(H)$ could
  not majorize~$\pi(G)$).

  We consider two cases.

  \smallskip
  \noindent
  \textbf{Case 1}. $|Y|\ge k$

  \noindent
  Since $|Y|>k-1=|E(X,Y)|$, some vertex $y\in Y$ is not incident to any
  edge in $E(X,Y)$. So $d_H(y)\le|Y|-1<\tfrac12n-1=\delta(G)$, and~$\pi(H)$
  would not majorize~$\pi(G)$.

  \smallskip
  \noindent
  \textbf{Case 2}. $|Y|\le k-1$

  \noindent
  Then any $y\in Y$ satisfies
  $d_H(y)\le|E(X,Y)|+(|Y|-1)\le(k-1)+(k-2)=2k-3= \tfrac12(4k)-3\le
  \tfrac12(n+2)-3=\tfrac12n-2<\tfrac12n-1=\delta(G)$. Again, $\pi$ would
  not majorize~$\pi(G)$.

  \medskip
  This proves the Claim, and completes the proof of Theorem~\ref{thm:313}.
\end{proof}

\noindent
In light of Theorem~\ref{thm:313}, it would be desirable to have a simple,
though not best monotone, condition for $k$-edge-connectedness that is at
least better than Theorem~\ref{thm:12} as a sufficient condition for
$k$-edge-connectedness. The following such condition was given
in~\cite{BHKS09}.

\begin{theorem}[(Bauer et al.\ \cite{BHKS09})]
  \label{thm:314}\mbox{}\\*
  Let $\pi=(d_1\le\cdots\le d_n)$ be a graphical sequence and let $k\ge1$
  be an integer. If
  \begin{subequations}
    \begin{gather}
      d_1\ge k;\label{eq:3111}\\
      d_{i-k+1}\le i-1\; \wedge\; d_i\le i+k-2\; \Rightarrow\;
      d_n\ge n-i+k-1,\quad
      \text{for $k+1\le i\le\bigl\lfloor{\tfrac12n}\bigr\rfloor$},
      \label{eq:3112}
    \end{gather}
  \end{subequations}
  then~$\pi$ is forcibly $k$-edge-connected.
\end{theorem}

\noindent
A sufficient condition for $k$-edge-connectedness stronger than
Theorem~\ref{thm:314} was given by Yin and Guo~\cite{HGTA}, though their
degree condition is substantially more involved than
conditions~\eqref{eq:3111} and~\eqref{eq:3112}. We refer the reader
to~\cite{HGTA} for details.

\setcounter{section}{3}
\setcounter{subsection}{1}
\setcounter{equation}{0}
\setcounter{theorem}{0}
\subsection{Binding Number}\label{sec3.2}

The concept of the binding number of a graph was first used by Anderson
\cite[p.~185]{A71}, and then given its present definition by
Woodall~\cite{W73}.

Given $S\subseteq V(G)$, let $N(S)\subseteq V(G)$ denote the neighbor set
of~$S$. Let
\[\mathcal{S}=\{\,S\subseteq V(G)\mid \text{$S\ne\varnothing$ and
  $N(S)\ne V(G)$}\,\}.\]
The \emph{binding number of~$G$}, denoted $\bind(G)$, is defined by
\[\bind(G)=\min_{S\in\mathcal{S}}\frac{|N(S)|}{|S|}.\]
In particular, $\bind(K_n)=n-1$, for $n\ge1$. A set $S\in\mathcal{S}$ for
which the above minimum is attained is called a \emph{binding set of~$G$}.
For $b\ge0$, we call $G$ \emph{$b$-binding} if $\bind(G)\ge b$.
Cunningham~\cite{C90} has shown that computing $\bind(G)$ is tractable.

A number of theorems in the literature guarantee that a graph~$G$ has a
certain property if $\bind(G)$ is appropriately bounded from below. The
following three theorems exhibit the best possible lower bound on
$\bind(G)$ to guarantee that~$G$ has the indicated property.

\begin{theorem}[(Anderson \cite{A71})]\label{thm:321}\mbox{}\\*
  If $|V(G)|$ is even and $\bind(G)\ge4/3$, then~$G$ contains a $1$-factor.
\end{theorem}

\begin{theorem}[(Woodall \cite{W73,W78})]\label{thm:322}\mbox{}\\*
  If $\bind(G)\ge3/2$, then~$G$ is hamiltonian.
\end{theorem}

\begin{theorem}[(Shi \cite{S87})]\label{thm:323}\mbox{}\\*
  If $\bind(G)\ge3/2$, then~$G$ contains a cycle of length~$l$, for
  $3\le l\le |V(G)|$.
\end{theorem}

\noindent
Other graph properties which are guaranteed by lower bounds on binding
number include the existence of an $f$-factor \cite{EE89,KW89,W90}, the
existence of a $k$-clique \cite{KM81,LP09}, and $k$-extendability
\cite{C95,RW00}.

In~\cite{BKSY11}, a best monotone condition was given for a graph to be
$b$-binding, first for $0<b\le1$ and then for $b\ge1$.

\begin{theorem}[(Bauer et al.\ \cite{BKSY11})]\label{thm:324}\mbox{}\\*
  Let $0<b\le1$ and let $\pi=(d_1\le\cdots\le d_n)$ be a graphical
  sequence, with $n\ge\lceil b+1\rceil=2$. If
  \begin{subequations}
    \begin{gather}
      d_i\le\lceil bi\rceil-1\; \Rightarrow\;
      d_{n-\lceil bi\rceil+1}\ge n-i,\quad
      \text{for $1\le i\le\Bigl\lfloor\frac{n}{b+1}\Bigr\rfloor$};
      \label{eq:321}\\
      d_{\lfloor n/(b+1)\rfloor+1}\ge
      n-\Bigl\lfloor\frac{n}{b+1}\Bigr\rfloor,\label{eq:322}
    \end{gather}
  \end{subequations}
  then~$\pi$ is forcibly $b$-binding.
\end{theorem}

\noindent
If~$\pi$ fails to satisfy~\eqref{eq:321} for some~$i$, then~$\pi$ is
majorized by the degrees of
$K_{\lceil bi\rceil-1}+
\bigl(K_{n-i-\lceil bi\rceil+1}\cup\overline{K_i}\bigr)$. If~$\pi$ fails to
satisfy~\eqref{eq:322}, then~$\pi$ is majorized by the degrees of
$K_{n-\lfloor n/(b+1)\rfloor-1}+\overline{K_{\lfloor n/(b+1)\rfloor+1}}$.
Since neither graph is $b$-binding, Theorem~\ref{thm:324} is weakly
optimal.

\begin{theorem}[(Bauer et al.\ \cite{BKSY11})]\label{thm:325}\mbox{}\\*
  Let $ b\ge1$, and let $\pi=(d_1\le\cdots\le d_n)$ be a graphical
  sequence, with $n\ge\lceil b+1\rceil$. If
  \begin{subequations}
    \begin{gather}
      d_i\le n-\Bigl\lfloor\frac{n-i}{b}\Big\rfloor-1\; \Rightarrow\;
      d_{\lfloor(n-i)/b\rfloor+1}\ge n-i,\quad
      \text{for $1\le i\le\Bigl\lfloor\frac{n}{b+1}\Bigr\rfloor$};
      \label{eq:323}\\
      d_{\lfloor n/(b+1)\rfloor+1}\ge
      n-\Bigl\lfloor\frac{n}{b+1}\Bigr\rfloor,\label{eq:324}
    \end{gather}
  \end{subequations}
  then~$\pi$ is forcibly $b$-binding.
\end{theorem}

\noindent
If~$\pi$ fails to satisfy~\eqref{eq:323} for some~$i$, then~$\pi$ is
majorized by the degrees of
$K_{n-\lfloor(n-i)/b\rfloor-1}+\linebreak
\bigl(K_{\lfloor(n-i)/b\rfloor-i+1}\cup\overline{K_i}\bigr)$. If~$\pi$
fails to satisfy~\eqref{eq:324} (which is the same as~\eqref{eq:322}),
then~$\pi$ is majorized by the degrees of
$K_{n-\lfloor n/(b+1)\rfloor-1}+\overline{K_{\lfloor n/(b+1)\rfloor+1}}$.
Since neither graph is $b$-binding, Theorem~\ref{thm:325} is weakly
optimal.

\setcounter{section}{3}
\setcounter{subsection}{2}
\setcounter{equation}{0}
\setcounter{theorem}{0}
\subsection{Toughness}\label{sec3.3}

The concept of toughness in graphs was introduced by Chv\'{a}tal
in~\cite{C73}. Let~$\omega(G)$ denote the number of components in a
graph~$G$. For $t\ge0$, we call~$G$ \emph{$t$-tough} if
$t\cdot\omega(G-X)\le|X|$, for every $X\subseteq V(G)$ with
$\omega(G-X)\ge2$. The \emph{toughness of~$G$}, denoted~$\tau(G)$, is the
maximum $t\ge0$ such that~$G$ is $t$-tough (taking $\tau(K_n)=n-1$, for
$n\ge1$). Thus if~$G$ is not complete, then
\[\tau(G)=\min\Bigl\{\,\frac{|X|}{\omega(G-X)}\Bigm|
\text{$X\subseteq V(G)$ with $\omega(G-X)\ge2$}\,\Bigr\}.\]
In~\cite{BHS90}, it was shown that computing~$\tau(G)$ is NP-hard.

Toughness has been especially prominent in connection with the existence of
long cycles in graphs. Indeed, it was a longstanding conjecture that every
$2$-tough graph is hamiltonian. But Bauer, Broersma, and
Veldman~\cite{BBV00} disproved this conjecture by constructing
$\bigl(\frac{9}{4}-\epsilon\bigr)$-tough nonhamiltonian graphs.
Unfortunately, the methods in~\cite{BBV00} do not extend to higher levels
of toughness, and it remains an open question whether there exists a
constant $t_0\ge9/4$ such that every $t_0$-tough graph is hamiltonian.

In~\cite{BBHKS13}, a best monotone condition for a graph to be $t$-tough
was given for $t \ge
 1$.

\begin{theorem}[(Bauer et al.\ \cite{BBHKS13})]\label{thm:331}\mbox{}\\*
  Let $t\ge1$, $n\ge\lceil t\rceil+2$, and $\pi=(d_1\le\cdots\le d_n)$ be a
  graphical sequence. If
  \begin{equation}\label{eq:331}
    d_{\lfloor i/t\rfloor}\le i\; \Rightarrow\;
    d_{n-i}\ge n-\lfloor i/t\rfloor,\quad
    \text{for $t\le i<\frac{tn}{t+1}$},
  \end{equation}
  then~$\pi$ is forcibly $t$-tough.
\end{theorem}

\noindent
If~$\pi$ fails to satisfy~\eqref{eq:331} for some~$i$, then~$\pi$ is
majorized by the degrees of
$K_i+\bigl(\overline{K_{\lfloor i/t\rfloor}}\cup
K_{n-i-\lfloor i/t\rfloor}\bigr)$ which is not $t$-tough. Thus
Theorem~\ref{thm:331} is weakly optimal. Note also that
condition~\eqref{eq:331} of Theorem~\ref{thm:331} reduces to
condition~\eqref{eq:11} of Theorem~\ref{thm:11} when $t=1$.

By Theorem~\ref{thm:331}, a best monotone $t$-tough condition on degree
sequences of length~$n$ requires fewer than~$n$ weakly optimal
Chv\'{a}tal-type conditions, for $t\ge1$. But this changes markedly as
$t\rightarrow0$. In particular, for any integer $k\ge2$, a best monotone
$\frac{1}{k}$-tough condition on degree sequences of length~$n$ requires at
least $f(k)n$ weakly optimal, Chv\'{a}tal-type conditions, where~$f(k)$
grows superpolynomially as $k\rightarrow\infty$. This is implied by
Theorem~\ref{thm:22} and the following result \cite[Lemma~4.2]{BBHKS13}.

\begin{theorem}[(Bauer et al.\ \cite{BBHKS13})]\label{thm:332}\mbox{}\\*
  Let $n=m(k+1)$, where $k\ge2$ and $m\ge9$ are integers. Then the number
  of $\frac{1}{k}$-tough sinks in~$G_n$ is at least
  $\dfrac{p(k-1)}{5(k+1)}n$, where~$p$ is the integer partition function.
\end{theorem}

\noindent
The superpolynomial growth in the complexity of a best monotone $t$-tough
theorem as $t\rightarrow0$ suggests the desirability of having a simple
$t$-tough theorem, for $0<t<1$. The following was given in~\cite{BBHKS13}.

\begin{theorem}[(Bauer et al.\ \cite{BBHKS13})]\label{thm:333}\mbox{}\\*
  Let $0<t<1$, $n\ge\lfloor1/t\rfloor+2$, and $\pi=(d_1\le\cdots\le d_n)$
  be a graphical sequence. If
  \begin{subequations}
    \begin{gather}
      d_i\le i-\lfloor 1/t\rfloor+1\; \Rightarrow\;
      d_{n-i+\lfloor 1/t\rfloor-1}\ge n-i,\quad
      \text{for $\lfloor 1/t\rfloor\le i<
        \tfrac12\bigl(n+\lfloor 1/t\rfloor-1\bigr)$};\label{eq:332}\\
      d_i\le i-1\; \Rightarrow\; d_n\ge n-i,\quad
      \text{for $1\le i\le\tfrac12n$},
    \end{gather}
  \end{subequations}
  then~$\pi$ is forcibly $t$-tough.
\end{theorem}

\setcounter{section}{3}
\setcounter{subsection}{3}
\setcounter{equation}{0}
\setcounter{theorem}{0}
\subsection{Factors}\label{sec3.4}

The \emph{deficiency} of a graph~$G$, denoted $\defi(G)$, is the number of
vertices unmatched under a maximum matching in~$G$. In particular, $G$ has
a $1$-factor if and only if $\defi(G)=0$. We call~$G$
\emph{$\beta$-deficient} if $\defi(G)\le\beta$.

In~\cite{V72} (see also~\cite{BC76}), a best monotone condition was given
for a graph to be $\beta$-deficient.

\begin{theorem}[(Las Vergnas \cite{V72})]\label{thm:341}\mbox{}\\*
  Let $\pi=(d_1\le\cdots\le d_n)$ be a graphical sequence, and let
  $0\le\beta\le n$ with $n\equiv\beta\pmod2$. If
  \begin{equation}\label{eq:341}
    d_{i+1}\le i-\beta\; \Rightarrow\; d_{n+\beta-i}\ge n-i-1,\quad
    \text{for $1\le i\le\tfrac12(n+\beta-2)$},
  \end{equation}
  then~$\pi$ is forcibly $\beta$-deficient.
\end{theorem}

\noindent
If~$\pi$ fails to satisfy~\eqref{eq:341} for some~$i$, then $\pi$ is
majorized by the degrees of
$K_{i-\beta}+\bigl(\overline{K_{i+1}}\cup K_{n-2i+\beta-1}\bigr)$, which is
not $\beta$-deficient. Thus Theorem~\ref{thm:341} is weakly optimal.

Taking $\beta=0$ in Theorem~\ref{thm:341}, we obtain a best monotone
condition for a graph to contain a $1$-factor.

In~\cite{BBHKS12}, a best monotone condition was given for a graph to
contain a $2$-factor.

\begin{theorem}[(Bauer et al.\ \cite{BBHKS12})]\label{thm:342}\mbox{}\\*
  Let $\pi=(d_1\le\cdots\le d_n)$ be a graphical sequence, with $n\ge3$. If
  (setting $d_0=0$)
  \begin{subequations}
    \begin{gather}
      d_{(n+1)/2}\ge\tfrac12(n+1),\quad \text{if $n$ is odd};
      \label{eq:342}\\
      d_{n/2-1}\ge\tfrac12n\; \vee\; d_{n/2+1}\ge\tfrac12n+1,\quad
      \text{if $n$ is even};\label{eq:343}\\
      d_i\le i\;\wedge\;d_{i+1}\le i+1\; \Rightarrow\;
      d_{n-i-1}\ge n-i-1\;\vee\;d_{n-i}\ge n-i,\quad
      \text{for $0\le i\le\tfrac12n-1$};\label{eq:344}\\
      d_{i-1}\le i\;\wedge\;d_{i+2}\le i+1\; \Rightarrow\;
      d_{n-i-3}\ge n-i-2\;\vee\;d_{n-i}\ge n-i-1,\quad
      \text{for $1\le i\le\tfrac12(n-5)$},\label{eq:345}
    \end{gather}
  \end{subequations}
  then~$\pi$ forcibly contains a $2$-factor.
\end{theorem}

\noindent
If~$\pi$ fails to satisfy~\eqref{eq:342}, then~$\pi$ is majorized by the
degrees of $K_{(n-1)/2}+\overline{K_{(n+1)/2}}$. If~$\pi$ fails to
satisfy~\eqref{eq:343}, then~$\pi$ is majorized by the degrees of
$K_{(n-2)/2}+\bigl(\overline{K_{(n-2)/2}}\cup K_2\bigr)$. If~$\pi$ fails to
satisfy~\eqref{eq:344} for some~$i$, then~$\pi$ is majorized by the degrees
of $K_i+\bigl(\overline{K_{i+1}}\cup K_{n-2i-1}\bigr)$ with an edge added
joining $\overline{K_{i+1}}$ and $K_{n-2i-1}$. If~$\pi$ fails to
satisfy~\eqref{eq:345} for some~$i$, then~$\pi$ is majorized by the degrees
of $K_i+\bigl(\overline{K_{i+2}}\cup K_{n-2i-2}\bigr)$ with three
independent edges joining $\overline{K_{i+2}}$ and $K_{n-2i-2}$. Since none
of these graphs contains a $2$-factor, Theorem~\ref{thm:342} is weakly
optimal.

We conjecture that the number of weakly optimal Chv\'{a}tal-type conditions
in a best monotone condition for a graph to contain a $k$-factor grows
rapidly with~$k$. More precisely, we put forth the following (cf.\
Theorem~\ref{thm:313} and Theorem~\ref{thm:332}).

\setcounter{conjecturesub}{\value{theorem}}
\begin{conjecturesub}\label{conj:343}\mbox{}\\*
  Let $f(k,n)$ denote the number of $k$-factor sinks in~$G_n$. Then there
  exist $a,b>0$ such that if $n\ge ak+b$, then $f(k,n)$ grows
  superpolynomially in~$k$.
\end{conjecturesub}

\setcounter{section}{3}
\setcounter{subsection}{4}
\setcounter{equation}{0}
\setcounter{theorem}{0}
\subsection{Paths and Cycles}\label{sec3.5}

In Section~\ref{sec1}, we noted that Theorem~\ref{thm:11} gives a best
monotone condition for hamiltonicity.

A graph~$G$ is \emph{$k$-hamiltonian} if for all $X\subseteq V(G)$ with
$|X|\le k$, the induced subgraph $\langle V(G)-X\rangle$ is hamiltonian.
Thus `$0$-hamiltonian' is the same as `hamiltonian'. A best monotone
condition for \mbox{$k$-hamiltonicity} was first given in~\cite{C72}
(although the form in which we present it is from~\cite{BC76}
and~\cite{L76}). Of course, Theorem~\ref{thm:11} is the special case $k=0$.

\begin{theorem}[(Chv\'atal \cite{C72})]\label{thm:351}\mbox{}\\*
  Let $\pi=(d_1\le\cdots\le d_n)$ be a graphical sequence with $n\ge3$, and
  let $0\le k\le n-3$. If
  \begin{equation}\label{eq:351}
    d_i\le i+k\; \Rightarrow\;  d_{n-i-k}\ge n-i,\quad
    \text{for $1\le i<\tfrac12(n-k)$},
  \end{equation}
  then~$\pi$ is forcibly $k$-hamiltonian.
\end{theorem}

\noindent
If~$\pi$ fails to satisfy~\eqref{eq:351} for some~$i$, then~$\pi$ is
majorized by the degrees of
$K_{i+k}+\bigl(\overline{K_i}\cup K_{n-2i-k}\bigr)$, which is not
$k$-hamiltonian. Thus Theorem~\ref{thm:351} is weakly optimal.

A graph is \emph{traceable} if it contains a hamiltonian path. A best
monotone condition for traceability was given in~\cite{C72}. More
generally, $G$ is \emph{$k$-path-coverable} if $V(G)$ can be covered by~$k$
or fewer vertex-disjoint paths. In particular, `$1$-path-coverable' is the
same as `traceable'. A best monotone condition for $k$-path-coverability
was obtained independently in~\cite{BC76} and~\cite{L76}.

\begin{theorem}[(Bondy \& Chv\'atal \cite{BC76}, Lesniak \cite{L76})]
  \label{thm:352}\mbox{}\\*
  Let $\pi=(d_1\le\cdots\le d_n)$ be a graphical sequence and let $k\ge1$.
  If
  \begin{equation}\label{eq:352}
    d_{i+k}\le i\; \Rightarrow\; d_{n-i}\ge n-i-k,\quad
    \text{for $1\le i<\tfrac12(n-k)$},
  \end{equation}
  then~$\pi$ is forcibly $k$-path-coverable.
\end{theorem}

\noindent
If~$\pi$ fails to satisfy~\eqref{eq:352} for some~$i$, then~$\pi$ is
majorized by the degrees of
$K_i+\bigl(\overline{K_{i+k}}\cup K_{n-2i-k}\bigr)$, which is not
$k$-path-coverable (adding~$k$ complete vertices to a graph which is
$k$-path coverable results in a hamiltonian graph, while adding~$k$
complete vertices to the above graph results in a graph which is not even
$1$-tough). Thus Theorem~\ref{thm:352} is weakly optimal.

A graph is \emph{hamiltonian-connected} if every pair of vertices is joined
by a hamiltonian path. A best monotone condition for
hamiltonian-connectedness was given in \cite[Chapter~10, Theorem~12]{B73}.

\begin{theorem}[(Berge \cite{B73})]\label{thm:353}\mbox{}\\*
  Let $\pi=(d_1\le\cdots\le d_n)$ be a graphical sequence with $n\ge4$. If
  \begin{equation}\label{eq:353}
    d_{i-1}\le i\; \Rightarrow\; d_{n-i}\ge n-i+1,\quad
    \text{for $2\le i<\tfrac12(n+1)$},
  \end{equation}
  then~$\pi$ is forcibly hamiltonian-connected.
\end{theorem}

\noindent
If~$\pi$ fails to satisfy~\eqref{eq:353} for some~$i$, then~$\pi$ is
majorized by the degrees of
$K_i+\bigl(\overline{K_{i-1}}\cup K_{n-2i+1}\bigr)$, which is not
hamiltonian-connected (there is no hamiltonian path joining two vertices
in~$K_i$). Thus Theorem~\ref{thm:353} is weakly optimal.

A graph~$G$ is \emph{$k$-edge-hamiltonian} if any collection of
vertex-disjoint paths with at most~$k$ edges altogether belong to a
hamiltonian cycle in~$G$. A best monotone condition for
$k$-edge-hamiltonicity was given in~\cite{K69} (see also \cite[Chapter~10,
Theorem~8]{B73}).

\begin{theorem}[(Kronk \cite{K69})]\label{thm:354}\mbox{}\\*
  Let $\pi=(d_1\le\cdots\le d_n)$ be a graphical sequence with $n\ge3$, and
  let $0\le k\le n-3$. If
  \begin{equation}\label{eq:354}
    d_{i-k}\le i\; \Rightarrow\; d_{n-i}\ge n-i+k,\quad
    \text{for $k+1\le i<\tfrac12(n+k)$},
  \end{equation}
  then~$\pi$ is forcibly $k$-edge-hamiltonian.
\end{theorem}

\noindent
If~$\pi$ fails to satisfy~\eqref{eq:354} for some~$i$, then~$\pi$ is
majorized by the degrees of
$K_i+\bigl(\overline{K_{i-k}}\cup K_{n-2i+k}\bigr)$, which is not
$k$-edge-hamiltonian (consider a path in~$K_i$ with~$k$ edges). Thus
Theorem~\ref{thm:354} is weakly optimal.

A graph~$G$ is \emph{pancyclic} if it contains an $l$-cycle for any~$l$
such that $3\le l\le|V(G)|$. We have the following best monotone condition
for~$\pi$ to be forcibly pancyclic.

\begin{theorem}\label{thm:355}\mbox{}\\*
  Let $\pi=(d_1\le\cdots\le d_n)$ be a graphical sequence with $n\ge3$. If
  \begin{subequations}
    \begin{gather}
      d_i\le i\; \Rightarrow\; d_{n-i}\ge n-i,\quad
      \text{for $1\le i<\tfrac12n$};\label{eq:355}\\
      d_n\ge\tfrac12n+1,\quad \text{if $n$ is even},\label{eq:356}
    \end{gather}
  \end{subequations}
  then~$\pi$ is forcibly pancyclic.
\end{theorem}

\begin{proof} In~\cite{BS90}, it is shown that if~$\pi$
  satisfies~\eqref{eq:355}, then~$G$ is pancyclic or bipartite. But if~$G$
  is bipartite, then, since~$G$ is hamiltonian by~\eqref{eq:355} and
  Theorem~\ref{thm:11}, $n$ is even and both bipartition sets have
  $\tfrac12n$ vertices. Thus $d_n\le\tfrac12n$, which
  contradicts~\eqref{eq:356}.
\end{proof}

\noindent
If~$\pi$ fails to satisfy~\eqref{eq:355} for some~$i$, then~$\pi$ is
majorized by the degrees of $K_i+\bigl(\overline{K_i}\cup K_{n-2i}\bigr)$,
which has no $n$-cycle. If~$\pi$ fails to satisfy~\eqref{eq:356},
then~$\pi$ is majorized by the degrees of $K_{n/2,n/2}$, which has no odd
length cycles. Thus Theorem~\ref{thm:355} is weakly optimal.

\setcounter{section}{3}
\setcounter{subsection}{5}
\setcounter{equation}{0}
\setcounter{theorem}{0}
\subsection{Independence Number, Clique Number, Chromatic Number, and
  Vertex Arboricity}\label{sec3.6}

We consider best monotone conditions for a graphical sequence~$\pi$ to be
forcibly $p(G)\le k$ or forcibly $p(G)\ge k$, where~$p$ denotes any of the
graph parameters $\alpha$~(independence number), $\omega$~(clique number),
$\chi$~(chromatic number), or $a$~(vertex arboricity). Note that if
$p\in\{\omega,\chi,a\}$, then $p(G)\le k$ is a decreasing property and
$p(G)\ge k$ is an increasing property (so that we seek upper bounds
on~$\pi$ in the first case and lower bounds in the second); while if
$p=\alpha$ then it is the other way around.

We begin with best monotone conditions for upper bounds $p(G)\le k$. We
consider first best monotone conditions for $\alpha(G)\le k$ and
$\omega(G)\le k$. Since
$\omega(G)\le k\Leftrightarrow\alpha(\overline{G})\le k$, the development
is analogous for~$\alpha$ and~$\omega$; we do only~$\alpha$.

\begin{theorem}\label{thm:361}\mbox{}\\*
  Let $\pi=(d_1\le\cdots\le d_n)$ be a graphical sequence, and $k\ge1$ an
  integer. If
  \begin{equation}
    d_{k+1}\ge n-k,\label{eq:361}
  \end{equation}
  then~$\pi$ is forcibly $\alpha(G)\le k$.
\end{theorem}

\begin{proof} Suppose~$\pi$ satisfies~\eqref{eq:361}, but has a
  realization~$G$ with $\alpha(G)\ge k+1$. If $S\subseteq V(G)$ is an
  independent set with $|S|\ge k+1$, then each vertex in~$S$ has degree at
  most $n-k-1$, and thus $d_{k+1}\le n-k-1$, contradicting~\eqref{eq:361}.
\end{proof}

\noindent
If~$\pi$ fails to satisfy~\eqref{eq:361}, then~$\pi$ is majorized by the
degrees of $G'=\overline{K_{k+1}}+K_{n-k-1}$, with $\alpha(G')=k+1>k$. Thus
Theorem~\ref{thm:361} is weakly optimal.

We also note that the optimal condition for $\alpha(G)\le k$ is tractable.
We begin with the following result of Rao~\cite{R79}.

\begin{theorem}[(Rao \cite{R79})]\label{thm:362}\mbox{}\\*
  A graphical sequence~$\pi$ has a realization $G$ with $\alpha(G)\ge k$ if
  and only if~$\pi$ has a realization in which vertices with the~$k$
  smallest degrees form an independent set.
\end{theorem}

\noindent
Using Theorem~\ref{thm:362}, it is easy to determine whether or not
$\pi=(d_1\le\cdots\le d_n)$ is forcibly $\alpha(G)\le k$: Iteratively
consider $k=1,2,\ldots,n-1$. To decide if~$\pi$ has a realization with
$k+1$ independent vertices, form the graph
$H=\overline{K_{k+1}}+K_{n-k-1}$, letting $v_1,\ldots,v_{k+1}$ (resp.,
$v_{k+2},\ldots,v_n$) denote the vertices of $\overline{K_{k+1}}$ (resp.,
$K_{n-k-1}$). Assign degree~$d_i$ to~$v_i$ for $1\le i\le n$, and determine
if~$H$ contains a subgraph~$H'$ with the assigned degrees. If so,
then~$\pi$ has a realization~$G$ with $\alpha(G)\ge k+1$, and~$\pi$ is not
forcibly $\alpha(G)\le k$. Otherwise, by Theorem~\ref{thm:362}, $\pi$ is
forcibly $\alpha(G)\le k$. Tutte~\cite{T54} proved the existence of~$H'$ is
equivalent to the existence of a $1$-factor in a graph that can be
efficiently constructed from~$H$ and $d_1,\ldots,d_n$.

Structural conditions guaranteeing $\chi(G)\le k$ (that~$G$ is
$k$-colorable) have a long and rich history \cite{B41,G68,SW68,W67}.
Regarding degree conditions, we first note the trivial bound
$\chi(G)\le\Delta(G)+1$, and thus

\begin{theorem}\label{thm:363}\mbox{}\\*
  The graphical sequence $\pi=(d_1\le\cdots\le d_n)$ is forcibly
  $\chi(G)\le d_n+1$.
\end{theorem}

\noindent
A best monotone condition for $\chi(G)\le k$ was given by Welsh and
Powell~\cite{WP67}.

\begin{theorem}[(Welsh \& Powell \cite{WP67})]\label{thm:364}\mbox{}\\*
  Let $\pi=(d_1\le\cdots\le d_n)$ be a graphical sequence. Then~$\pi$ is
  forcibly
  \[\chi(G)\le\max_{1\le j\le n}\min\{\,n-j+1,\,d_j+1\,\}.\]
\end{theorem}

\noindent
Reexpressing Theorem~\ref{thm:364} with an equivalent Chv\'{a}tal-type
degree condition, we have the following.

\begin{theorem}\label{thm:365}\mbox{}\\*
  Let $\pi=(d_1\le\cdots\le d_n)$ be a graphical sequence, and let
  $1\le k\le n$. If (setting $d_0=0$)
  \begin{equation}\label{eq:362}
    d_{n-k}\le k-1,
  \end{equation}
  then~$\pi$ is forcibly $\chi(G)\le k$.
\end{theorem}

\noindent
If~$\pi$ fails to satisfy~\eqref{eq:362}, then~$\pi$ is minorized by the
vertex degrees of $G=K_{k+1}\cup\overline{K_{n-k-1}}$, with
$\chi(G)=k+1>k$. Thus Theorem~\ref{thm:365} is weakly optimal.

Analogous to the bound $\chi(G)\le\Delta(G)+1$, we have
$a(G)\le\bigl\lfloor\tfrac12\Delta(G)\bigr\rfloor+1$~\cite{CK69}, and thus
we get the following.

\begin{theorem}\label{thm:366}\mbox{}\\*
  Let $\pi=(d_1\le\cdots\le d_n)$ be a graphical sequence. Then~$\pi$ is
  forcibly $a(G)\le\bigl\lfloor\tfrac12d_n\bigr\rfloor+1$.
\end{theorem}

\noindent
A best monotone condition for $a(G)\le k$ was given in~\cite{HS89}; it is
analogous to Theorem~\ref{thm:364}.

\begin{theorem}[(Hakimi \& Schmeichel \cite{HS89})]\label{thm:367}
  \mbox{}\\*
  Let $\pi=(d_1\le\cdots\le d_n)$ be a graphical sequence. Then~$\pi$ is
  forcibly
  \[a(G)\le\max_{1\le j\le n}\min\bigl\{\,
  \bigl\lceil\tfrac12(n-j+1)\bigr\rceil,\,
  \bigl\lceil\tfrac12(d_j+1)\bigr\rceil\,\bigr\}.\]
\end{theorem}

\noindent
Reexpressing Theorem~\ref{thm:367} with an equivalent Chv\'{a}tal-type
degree condition, we have the following.

\begin{theorem}\label{thm:368}\mbox{}\\*
  Let $\pi=(d_1\le\cdots\le d_n)$ be a graphical sequence, and let
  $1\le k\le\tfrac12n$. If
  \begin{equation}\label{eq:363}
    d_{n-2k}\le 2k-1,
  \end{equation}
  then~$\pi$ is forcibly $a(G)\le k$.
\end{theorem}

\noindent
If~$\pi$ fails to satisfy~\eqref{eq:363}, then~$\pi$ is minorized by the
vertex degrees of $G=K_{2k+1}\cup\overline{K_{n-2k-1}}$, with $a(G)=k+1>k$.
Thus Theorem~\ref{thm:368} is weakly optimal.

We turn next to best monotone conditions for lower bounds $p(G)\ge k$. The
most prominent degree condition for $\alpha(G)\ge k$, although not best
monotone, is independently due to Caro~\cite{C79} and Wei~\cite{W81}. An
elegant probabilistic proof appears in \cite[p.~81]{AS92} (see also
\cite[p.~428]{W01}).

\begin{theorem}[(Caro \cite{C79}, Wei \cite{W81})]\label{thm:369}
  \mbox{}\\*
  Let $\pi=(d_1\le\cdots\le d_n)$ be a graphical sequence. Then~$\pi$ is
  forcibly $\alpha(G)\ge\sum\limits_{j=1}^n\dfrac{1}{d_j +1}$.
\end{theorem}

\noindent
A best monotone condition for $\alpha(G)\ge k$ was given by
Murphy~\cite{M91}. Let $\pi=(d_1\le\cdots\le d_n)$ be a graphical sequence.
Define $f:\mathbb{Z}^+\rightarrow\{d_1,d_2,\ldots,d_n,\infty\}$ recursively
as follows: Set $f(1)=d_1$. If $f(i)=d_j$, then set
\[f(i+1)=\begin{cases}
  d_{j+f(i)+1},&\text{if $j+f(i)+1\le n$};\\
  \infty,&\text{otherwise};
\end{cases}\]
while if $f(i)=\infty$, then $f(i+1)=\infty$.

Murphy's condition is the following.

\begin{theorem}[(Murphy \cite{M91})]\label{thm:3610}\mbox{}\\*
  Let $\pi=(d_1\le\cdots\le d_n)$ be a graphical sequence. Then~$\pi$ is
  forcibly $\alpha(G)\ge\max\{\,i\in\mathbb{Z}^+\mid\linebreak[1]
  f(i)<\infty\,\}$.
\end{theorem}

\noindent
\textbf{Example}. If $\pi=1^54^26^27^3$, then $f(1)=d_1=1$, $f(2)=d_3=1$,
$f(3)=d_5=1$, $f(4)=d_7=4$, $f(5)=d_{12}=7$, and $f(6)=f(7)=\cdots=\infty$.
The calculation of the $f(i)$ can be nicely visualized, as shown below.

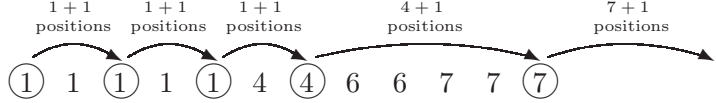
\begin{center}
  \begin{picture}(300,45)
    \put(11,5){$1 \; \; \; \; 1 \; \; \; \; 1 \; \; \; \;  1 \; \; \; \;  1 \; \; \; \;  4 \; \; \; \; 4 \; \; \; \; 6 \; \; \; \; 6 \; \; \; \; 7 \; \; \; \; 7 \; \; \; \; 7$} \put(13.6,9){\circle{13}}
    \put(49.5,9){\circle{13}} \put(84.5,9){\circle{13}}
    \put(120,9){\circle{13}} \put(208,9){\circle{13}}
    \thinlines \qbezier(16,18)(31.5,28)(47,18)
    \qbezier(52,18)(67.5,28)(83,18) \qbezier(88,18)(103,28)(117,18)
    \qbezier(123,18)(165,28)(205,18) \qbezier(211,18)(239,28)(271,18)
    \put(22,35){{\tiny $1+1$}} \put(17,28){{\tiny positions}}
    \put(58,35){{\tiny $1+1$}} \put(53,28){{\tiny positions}}
    \put(94,35){{\tiny $1+1$}} \put(89,28){{\tiny positions}}
    \put(155,35){{\tiny $4+1$}} \put(150,28){{\tiny positions}}
    \put(233,35){{\tiny $7+1$}} \put(228,28){{\tiny positions}}
    \thicklines \put(47,18){\vector(2,-1){3}}
    \put(83,18){\vector(2,-1){3}} \put(117,18){\vector(2,-1){3}}
    \put(205,18){\vector(2,-1){3}} \put(271,18){\vector(2,-1){3}}
  \end{picture}
\end{center}

\noindent
So by Theorem~\ref{thm:3610}, $\pi$ is forcibly $\alpha(G)\ge5$ (the number
of circled vertices).\qquad\hspace*{\fill}$\triangle$

\medskip
\noindent
By comparison, Theorem~\ref{thm:369} guarantees only that~$\pi$ in the
above example is forcibly $\alpha(G)\ge4$. Indeed, Theorem~\ref{thm:3610}
can be arbitrarily better than Theorem~\ref{thm:369}: For the graphical
sequence $\pi=1^12^23^3\cdots d^d$ with $d\equiv 0\pmod4$,
Theorem~\ref{thm:369} (resp., Theorem~\ref{thm:3610}) guarantees that~$\pi$
is forcibly
$\alpha(G)\ge\sum\limits_{i=1}^d\Bigl(1-\dfrac{1}{i+1}\Bigr)\sim d-\ln d$
(resp., $\alpha(G)\ge d$).

To see that Theorem~\ref{thm:3610} is weakly optimal for $\alpha(G)\ge k$,
suppose Theorem~\ref{thm:3610} fails to guarantee that~$\pi$ is forcibly
$\alpha(G)\ge k$. Consider the degree sequence
$\pi'=f(1)^{f(1)+1}f(2)^{f(2)+1}\linebreak[1]\cdots\linebreak[1]
f(k-2)^{f(k-2)+1}l^{l+1}$, where~$l$ denotes the number of degrees in $\pi$
with index greater than $f(k-1)$. Note that $l\le f(k-1)$, since
$l\ge f(k-1)+1$ implies $f(k)<\infty$, contradicting that
Theorem~\ref{thm:3610} does not declare~$\pi$ forcibly $\alpha(G)\ge k$.
Thus~$\pi'$ minorizes~$\pi$. But~$\pi'$ has realization
$G'=K_{f(1)+1}\cup\cdots\cup K_{f(k-2)+1}\cup K_{l+1}$ consisting of $k-1$
disjoint cliques, with $\alpha(G')=k-1<k$. Thus Theorem~\ref{thm:3610} is
weakly optimal for $\alpha(G)\ge k$.

Using Theorem~\ref{thm:3610}, we can easily obtain best monotone conditions
for $\omega(G)\ge k$ and $\chi(G)\ge k$. Let $\pi=(d_1\le\cdots\le d_n)$ be
a graphical sequence, and define
$g(\pi)=\max\{\,i\in\mathbb{Z}^+\mid f(i)<\infty\,\}$ as above (so
that~$\pi$ is forcibly $\alpha(G)\ge g(\pi)$). Define
$h:\{\,\text{Graphical Sequences}\,\}\rightarrow\mathbb{Z}^+$ by
$h(\pi)=g(\overline{\pi})$, where
$\overline{\pi}=((n-1)-d_n\le\cdots\le (n-1)-d_1)$ is the degree sequence
complementary to~$\pi$. If $G,\overline{G}$ are arbitrary realizations of
$\pi,\overline{\pi}$, then
$h(\pi)=g(\overline{\pi})\le\alpha(\overline{G})=\omega(G)\le\chi(G)$.
Since~$g$ is monotone decreasing, $h$ is monotone increasing. So to prove
that $\omega(G)\ge h(\pi)$ and $\chi(G)\ge h(\pi)$ are best monotone lower
bounds, it suffices to show these lower bounds are weakly optimal. But if
$h(\pi)= g(\overline{\pi})\le k-1$, then as above there exists
a~$\overline{\pi}'$ minorizing~$\overline{\pi}$ with a
realization~$\overline{G}'$ consisting of $k-1$ disjoint cliques. So
$\pi'\ge\pi$ has a realization $G'=\overline{\overline{G}'}$ that is a
complete $(k-1)$-partite graph. Thus $\omega(G'),\chi(G')=k-1<k$, and the
above lower bounds for~$\omega$ and~$\chi$ are weakly optimal.

Finally, if~$P$ is the property $a(G)\ge k$, it was proved in~\cite{HNVAVD}
that $|S(P,n)|$ grows superpolynomially in~$n$. Indeed,
$|S(P,n)|\ge D\Bigl(\Bigl\lfloor\dfrac{n}{k-1}\Bigr\rfloor\Bigr)$ if $k-1$
divides $n$, where~$D(m)$ denotes the number of different degree sequences
of unlabeled $m$-vertex trees. We refer the reader to~\cite{HNVAVD} for
details. But $D(m)=p(m-2)$ \cite[Chapter~6, Theorem~8]{B73} grows
superpolynomially in~$m$, where~$p$ is the integer partition function.
Thus, any best monotone condition for~$P$ will be inherently complex.

\setcounter{section}{3}
\setcounter{equation}{0}
\setcounter{theorem}{0}
\renewcommand{\thetheorem}{\arabic{section}.\arabic{theorem}}
\section{Best Monotone Degree Conditions for Implications
  \boldmath$P_1\Rightarrow P_2$}\label{sec4}

In this section, we consider best monotone degree conditions for
$P_1\Rightarrow P_2$, where $P_1,P_2$ are monotone increasing graph
properties. Conditions of this type were first considered
in~\cite{BNSWY14}. A framework for such considerations is given in
Section~\ref{sec2} upon substituting $P_1\Rightarrow P_2$ for~$P$
throughout. Best monotone $P_1\Rightarrow P_2$ conditions are particularly
interesting when~$P_1$ is a necessary condition for~$P_2$, since they
provide the minimum degree strength which needs to be added to~$P_1$ to get
a sufficient condition for~$P_2$.

In this section, we will focus on best monotone degree conditions for
$P_1\Rightarrow P_2$ -- some proved, some conjectured -- when~$P_2$ is
`hamiltonian', and~$P_1$ belongs to the set \{\,`traceable',
`$2$-connected', `$1$-binding', `contains a $2$-factor', `$1$-tough'\,\} of
well-known necessary conditions for hamiltonicity.

We begin with two results which are essentially immediate corollaries of
Theorem~\ref{thm:11}.

\begin{theorem}\label{thm:41}\mbox{}\\*
  Let $\pi=(d_1\le\cdots\le d_n)$ be a graphical sequence, with $n\ge3$. If
  \begin{equation}\label{eq:41}
    d_i\le i\; \Rightarrow\; d_{n-i}\ge n-i,\quad
    \text{for $1\le i\le\tfrac12(n-1)$},
  \end{equation}
  then every traceable realization of~$\pi$ is hamiltonian.
\end{theorem}

\noindent
Note that~\eqref{eq:41} is the same degree condition as in
Theorem~\ref{thm:11}. If~$\pi$ fails to satisfy~\eqref{eq:41} for some~$i$,
then~$\pi$ is majorized by the degrees of
$K_i+\bigl(\overline{K_i}\cup K_{n-2i}\bigr)$, which is traceable and
nonhamiltonian. Thus Theorem~\ref{thm:41} is weakly optimal for
$\text{traceable}\Rightarrow\text{hamiltonian}$.

\begin{theorem}\label{thm:42}\mbox{}\\*
  Let $\pi=(d_1\le\cdots\le d_n)$ be a graphical sequence, with $n\ge3$. If
  \begin{equation}\label{eq:42}
    d_i\le i\; \Rightarrow\; d_{n-i}\ge n-i,\quad
    \text{for $2\le i\le\tfrac12(n-1)$},
  \end{equation}
  then every $2$-connected realization of~$\pi$ is hamiltonian.
\end{theorem}

\noindent
If~$\pi$ fails to satisfy~\eqref{eq:42} for some~$i$, then~$\pi$ is
majorized by the degree sequence of
$K_i+\bigl(\overline{K_i}\cup K_{n-2i}\bigr)$, which is $2$-connected
(since $i\ge2$) and nonhamiltonian. Thus Theorem~\ref{thm:42} is weakly
optimal for $\text{$2$-connected}\Rightarrow\text{hamiltonian}$.

We have the following best monotone condition for
$\text{$1$-binding}\Rightarrow\text{hamiltonian}$~\cite{BKSWY1}.

\begin{theorem}[(Bauer et al.\ \cite{BKSWY1})]\label{thm:43}\mbox{}\\*
  Let $\pi=(d_1\le\cdots\le d_n)$ be a graphical sequence, with $n\ge3$. If
  (setting $d_0=0$)
  \begin{subequations}
    \begin{gather}
      d_i\le i\; \Rightarrow\; d_{n-i}\ge n-i,\quad
      \text{for $1\le i\le\tfrac12(n-1)$};\label{eq:43}\\
      d_{i-1}\le i\; \Rightarrow\; d_{n-i}\ge\tfrac12(n+1),\quad
      \text{for $1\le i\le\tfrac12(n-3)$, if $n$ is odd},\label{eq:44}
    \end{gather}
  \end{subequations}
  then every $1$-binding realization of~$\pi$ is hamiltonian.
\end{theorem}

\noindent
If~$\pi$ fails to satisfy~\eqref{eq:43} (resp.,~\eqref{eq:44}) for
some~$i$, then~$\pi$ is majorized by the degrees of
$K_i+\linebreak[1] \bigl(\overline{K_i}\cup K_{n-2i}\bigr)$ (resp.,
$K_i+\bigl(\overline{K_{i-1}}\cup 2K_{(n+1)/2-i}\bigr)$), which are each
$1$-binding and nonhamiltonian. Thus Theorem~\ref{thm:43} is weakly optimal
for $\text{$1$-binding}\Rightarrow\text{hamiltonian}$.

Since $1$-binding is also a necessary condition for a graph to contain a
$1$-factor, the following is of some interest~\cite{BNS}.

\begin{theorem}[(Bauer, Nevo \& Schmeichel \cite{BNS})]\label{thm:44}
  \mbox{}\\*
  Let $\pi=(d_1\le\cdots\le d_n)$ be a graphical sequence, with~$n$ even.
  If (setting $d_0=0$)
  \begin{subequations}
    \begin{gather}
      \begin{array}{@{}r@{}}
        d_i\le i\;\wedge\;d_{i+2j+1}\le i+2j\; \Rightarrow\;
        d_{n-i}\ge n-(i+2j+1),\qquad\qquad\qquad\qquad\qquad\qquad\qquad\\
        \text{for $1\le i\le\tfrac12(n-6)$ and
          $1\le j\le\tfrac14(n-2i-2)$};
      \end{array}\label{eq:45}\\
      d_{n/2-5}\ge\tfrac12n-3\; \vee\; d_{n/2+4}\ge\tfrac12n-1,\quad
      \text{if $n\ge10$},\label{eq:46}
    \end{gather}
  \end{subequations}
  then every $1$-binding realization of~$\pi$ contains a $1$-factor.
\end{theorem}

\noindent
If~$\pi$ fails to satisfy~\eqref{eq:45} for some $i,j$
(resp.,~\eqref{eq:46}), then~$\pi$ is majorized by the degree sequence of
$K_i+\bigl(\overline{K_i}\cup K_{2j+1}\cup K_{n-2i-2j-1}\bigr)$ (resp.,
$K_{n/2-4}+\bigl(\overline{K_{n/2-5}}\cup3K_3\bigr)$) which are each
$1$-binding without a $1$-factor. Thus Theorem~\ref{thm:44} is weakly
optimal for $\text{$1$-binding}\Rightarrow\text{$1$-factor}$.

For $1<b<3/2$, a best monotone degree condition for
$\text{$b$-binding}\Rightarrow\text{hamiltonian}$ is not currently known.
An asymptotically best minimum degree condition for
$\text{$b$-binding}\Rightarrow\text{hamiltonian}$ when $1<b<3/2$, namely
$\delta(G)\ge\Bigl(\dfrac{2-b}{3-b}\Bigr)n$, was established
in~\cite{BS12}. A somewhat involved best monotone degree condition for
$\text{$b$-binding}\Rightarrow\text{$1$-tough}$ was given in~\cite{BKSWY1},
where it was conjectured that this condition is also a best monotone degree
condition for $\text{$b$-binding}\Rightarrow\text{hamiltonian}$. We refer
the reader to~\cite{BKSWY1} for details.

A best monotone degree condition for
$\text{$2$-factor}\Rightarrow\text{hamiltonian}$ is also not currently
known. As the graph $K_1+2K_{(n-1)/2}$ shows, a best minimum degree
condition for $\text{$2$-factor}\Rightarrow\text{hamiltonian}$ is Dirac's
hamiltonian condition $\delta(G)\ge\tfrac12n$. On the other hand, we have
the following best monotone degree condition for
$\text{$2$-factor}\Rightarrow\text{$1$-tough}$~\cite{HNBMC2F1T}.

\begin{theorem}[(Bauer, Nevo \& Schmeichel \cite{HNBMC2F1T})]
  \label{thm:45}\mbox{}\\*
  Let $\pi=(d_1\le\cdots\le d_n)$ be a graphical sequence, with $n\ge3$. If
  (setting $d_0=0$)
  \begin{subequations}
    \begin{gather}
      d_i\le i\; \Rightarrow\; d_{n-i}\ge n-i,\quad
      \text{for $1\le i\le\tfrac12(n-3)$};\label{eq:47}\\
      d_{i-1}\le i\; \Rightarrow\; d_{n-i}\ge\tfrac12(n+1),\quad
      \text{for $1\le i\le\tfrac12(n-5)$, if $n$ is odd};\label{eq:48}\\
      d_{i-1}\le i\; \Rightarrow\; d_{n/2-1}\ge\tfrac12n\; \vee\;
      d_{n-i}\ge\tfrac12n+1,\quad
      \text{for $1\le i\le\tfrac12(n-4)$, if $n$ is even},
      \label{eq:49}
    \end{gather}
  \end{subequations}
  then every realization of~$\pi$ with a $2$-factor is $1$-tough.
\end{theorem}

\noindent
If~$\pi$ fails to satisfy \eqref{eq:47}, \eqref{eq:48}, or~\eqref{eq:49},
resp., for some~$i$, then~$\pi$ is majorized by the degree sequence of
$K_i+\bigl(\overline{K_i}\cup K_{n-2i}\bigr)$,
$K_i+\bigl(\overline{K_{i-1}}\cup 2K_{(n+1)/2-i}\bigr)$, or
$K_i+\bigl(\overline{K_{i-1}}\cup K_{n/2-i}\cup K_{n/2+1-i}\bigr)$, resp.,
where each graph contains a $2$-factor, but is not $1$-tough. Thus
Theorem~\ref{thm:45} is weakly optimal for
$\text{$2$-factor}\Rightarrow\text{$1$-tough}$.

We put forth the following conjecture.

\setcounter{conjecture}{\value{theorem}}
\addtocounter{theorem}{1}
\begin{conjecture}\label{conj:46}\mbox{}\\*
  The degree condition in Theorem~\ref{thm:45} is a best monotone degree
  condition for $\text{$2$-factor}\Rightarrow\linebreak\text{hamiltonian}$.
\end{conjecture}

\noindent
A best monotone condition for
$\text{$1$-tough}\Rightarrow\text{hamiltonian}$ is again not currently
known. However, a best minimum degree condition for
$\text{$1$-tough}\Rightarrow\text{hamiltonian}$ was given by
Jung~\cite{J78} (see also~\cite{BMS89}).

\begin{theorem}[(Jung \cite{J78})]\label{thm:47}\mbox{}\\*
  Let~$G$ be a $1$-tough graph on $n\ge11$ vertices. If
  $\delta(G)\ge\tfrac12n-2$, then~$G$ is hamiltonian.
\end{theorem}

\noindent
In~\cite{H95}, Ho\`{a}ng gave the following simple, but not best monotone,
degree condition for $\text{$1$-tough}\Rightarrow\text{hamiltonian}$, and
noted the difficulty of determining the
$\text{$1$-tough}\Rightarrow\text{hamiltonian}$ sinks.

\begin{theorem}[(Ho\`{a}ng \cite{H95})]\label{thm:48}\mbox{}\\*
  Let $\pi=(d_1\le\cdots\le d_n)$ be a graphical sequence, with $n\ge3$. If
  \begin{equation*}
    d_i\le i\;\wedge\;d_{n-i+1}\le n-i-1\; \Rightarrow\;
    d_j+d_{n-j+1}\ge n,\quad
    \text{for $1\le i<j\le\bigl\lceil\tfrac12n\bigr\rceil$},\label{eq:410a}
  \end{equation*}
  then every $1$-tough realization of~$\pi$ is hamiltonian.
\end{theorem}

\setcounter{corollary}{\value{theorem}}
\addtocounter{theorem}{1}
\begin{corollary}\label{cor:49}\mbox{}\\*
  Let $\pi=(d_1\le\cdots\le d_n)$ be a graphical sequence, with $n\ge3$. If
  \begin{equation*}
    d_i\le i\; \Rightarrow\; d_{n-i+1}\ge n-i,\quad
    \text{for $1\le i\le\tfrac12(n-1)$},\label{eq:411a}
  \end{equation*}
  then every $1$-tough realization of~$\pi$ is hamiltonian.
\end{corollary}

\noindent
A best monotone condition for $1$-tough $\Rightarrow$ $2$-factor was given
in~\cite{HNBMC1T2F}.

\begin{theorem}[(Bauer, Nevo \& Schmeichel \cite{HNBMC1T2F})]
  \label{thm:410}\mbox{}\\*
  Let $\pi=(d_1\le\cdots\le d_n)$ be a graphical sequence, with $n\ge3$. If
  (setting $d_0=0$)
  \begin{subequations}
    \begin{gather}
      \begin{array}{@{}r@{}}
        d_i\le i+j\;\wedge\;d_{i+2j+1}\le i+j+1\;\Rightarrow\;
        d_{n-i-3j-1}\ge n-i-2j-1\;\vee\;d_{n-i-j}\ge n-i-2j,\quad\\
        \text{for $0\le i\le\tfrac12(n-7)$ and
          $1\le j\le\tfrac15(n-2i-2)$};
      \end{array}\label{eq:410}\\
      \begin{array}{@{}r@{}}
        d_i\le i+2\;\wedge\;d_{i+4}\le i+3\;\Rightarrow\;
        d_{n-i-6}\ge\tfrac12n-1\;\vee\;d_{n-i-2}\ge\tfrac12n,
        \qquad\qquad\qquad\qquad\qquad\quad\\
        \text{for $0\le i\le\tfrac12(n-18)$, if $n\ge18$ is even};
      \end{array}\label{eq:411}\\
      \begin{array}{@{}r@{}}
        d_i\le i +1\;\wedge\;d_{i+2}\le i+2\;\wedge\;d_{i+3}\le i+3\;
        \Rightarrow\; d_{n-i-5}\ge\tfrac12n-1\;\vee\;
        d_{n-i-1}\ge\tfrac12n,\qquad\qquad\\
        \text{for $0\le i\le\tfrac12(n-16)$, if $n\ge16$ is even};
      \end{array}\label{eq:412}\\
      d_{n/2-5}\ge\tfrac12n-2\; \vee\; d_{n/2}\ge\tfrac12n-1\; \vee\;
      d_{n/2+3}\ge\tfrac12n+1,\quad \text{if $n\ge10$ is even},
      \label{eq:413}
    \end{gather}
  \end{subequations}
  then every $1$-tough realization of~$\pi$ contains a $2$-factor.
\end{theorem}

\noindent
If~$\pi$ fails to satisfy~\eqref{eq:410} for some $i,j$, then~$\pi$ is
majorized by the degrees of $K_{i+j}+\bigl(\overline{K_{i+2j+1}}\cup
K_{n-2i-3j-1}\bigr)$ with $\overline{K_{i+2j+1}}$ and $K_{n-2i-3j-1}$
joined by $2j+1$ independent edges. If~$\pi$ fails to
satisfy~\eqref{eq:411} for some~$i$, then~$\pi$ is majorized by the degrees
of $K_{i+2}+\bigl(\overline{K_{i+4}}\cup 2K_{n/2-i-3}\bigr)$ with
$\overline{K_{i+4}}$ joined to the two copies of $K_{n/2 -i-3}$ by four
independent edges, three of the edges to one copy, and one to the other.
If~$\pi$ fails to satisfy~\eqref{eq:412} for some~$i$, then~$\pi$ is
majorized by the degrees of
$K_{i+1}+\bigl(\overline{K_{i+3}}\cup2K_{n/2-i-2}\bigr)$ with $x,y,z\in
V(\overline{K_{i+3}})$ joined by three independent edges to one copy of
$K_{n/2-i-2}$, and~$x$ joined by one edge to the other copy. If~$\pi$ fails
to satisfy~\eqref{eq:413}, then~$\pi$ is majorized by the degree sequence
of $K_{n/2-3}+\bigl(\overline{K_{n/2-1}}\cup K_3\cup K_1\bigr)$ with
$\overline{K_{n/2-1}}$ joined by four independent edges to $K_3\cup K_1$.
Each of these graphs is $1$-tough and does not contain a $2$-factor. Thus
Theorem~\ref{thm:410} is weakly optimal for
$\text{$1$-tough}\Rightarrow\text{$2$-factor}$.

We conclude this section with the following question.

\begin{query}\label{qe4:1}\mbox{}\\*
  Is the degree condition in Theorem~\ref{thm:410} also a best monotone
  degree condition for $\text{$1$-tough}\Rightarrow\text{hamiltonian}$?
\end{query}

\setcounter{section}{4}
\setcounter{equation}{0}
\setcounter{theorem}{0}
\section{Improving Structural Results in a Best Monotone Sense}\label{sec5}

Recall from Section~\ref{sec2} that if~$P$ is a graph property and~$\pi$ is
a graphical sequence, then $\pi\in\BM(P)$ if and only if every graphical
sequence $\pi'\ge\pi$ is forcibly~$P$. Let $P_1,P_2$ be graph properties.
If $P_1\Rightarrow P_2$ and $\pi\in\BM(P_1)$, then~$\pi$ is forcibly~$P_2$.
But more is true \cite{BKSWY2}.

\begin{theorem}[(Bauer et al.\ \cite{BKSWY2})]\label{thm:51}\mbox{}\\*
  Let $P_1,P_2$ be graph properties such that $P_1\Rightarrow P_2$. Then
  $\pi\in\BM(P_1)\Rightarrow \pi\in\BM(P_2)$.
\end{theorem}

\noindent
In the remainder of this section we will abbreviate
$\pi\in\BM(P_1)\Rightarrow\pi\in\BM(P_2)$ (equivalent to
$\BM(P_1)\subseteq\BM(P_2)$) by $\BM(P_1)\Rightarrow\BM(P_2)$.

For example, since $\text{$3/2$-binding}\Rightarrow\text{hamiltonian}$ by
Theorem~\ref{thm:322}, we have
\begin{equation}\label{eq:51}
  \text{BM($3/2$-binding)}\Rightarrow\text{BM(hamiltonian)}
\end{equation}
by Theorem~\ref{thm:51}. We may think of~\eqref{eq:51} as a best monotone
analogue of the structural implication
$\text{$3/2$-binding}\Rightarrow\text{hamiltonian}$.

Our interest in this section will be in implications of the form
$\BM(P_1)\Rightarrow\BM(P_2)$ when the analogous structural implication
$P_1\Rightarrow P_2$ does \underline{not} hold. In that case, we will call
$\BM(P_1)\Rightarrow\BM(P_2)$ an \emph{improvement in a best monotone
  sense} of the structural $P_1\not\Rightarrow P_2$. In the remainder of
this section, we illustrate the possibility of obtaining such improvements
with several examples.

\medskip
\noindent
\textbf{1)}\quad Although $\text{hamiltonian}\Rightarrow\text{$1$-tough}$,
the converse $\text{$1$-tough}\Rightarrow\text{hamiltonian}$ fails to hold.
But the converse does hold in a best monotone sense, i.e.,
$\BM(\text{$1$-tough})\Rightarrow\BM(\text{hamiltonian})$, by
Theorem~\ref{thm:331} ($t=1$) and Theorem~\ref{thm:11}.

\medskip
\noindent
\textbf{2)}\quad Although $a(G)\le k\Rightarrow\chi(G)\le 2 k$, the
converse $\chi(G)\le2k\Rightarrow a(G)\le k$ is not true. But the converse
is true in a best monotone sense, i.e.,
$\BM(\chi(G)\le2k)\Rightarrow\BM(a(G)\le k)$, by Theorem~\ref{thm:365} and
Theorem~\ref{thm:368}.

\medskip
\noindent
\textbf{3)}\quad As noted in Section~\ref{sec3.2}, the bound
$\bind(G)\ge3/2$ in Theorem~\ref{thm:322} is best possible, and thus there
is no structural implication of the form
$\text{$b$-binding}\Rightarrow\text{hamiltonian}$, for any $b<3/2$. But
this implication does hold in a best monotone sense for
$b>1$~\cite{BKSY11}.

\begin{theorem}[(Bauer et al. \cite{BKSY11})]\label{thm:52}\mbox{}\\*
  If $b>1$, then
  $\BM(\text{$b$-binding})\Rightarrow\BM(\text{hamiltonian})$.
\end{theorem}

\noindent
The hypothesis $b>1$ in Theorem~\ref{thm:52} is best possible: If
\[\pi=\Bigl(\bigl\lfloor\tfrac12n\bigr\rfloor-1
\Bigr)^{\lfloor n/2\rfloor-1}
\Bigr(n-\bigl\lfloor\tfrac12n\bigr\rfloor\Bigr)^{n-2\lfloor n/2\rfloor+2}
\bigl(n-1\bigr)^{\lfloor n/2\rfloor-1},
\end{equation*}
then $\pi\in\BM(\text{$1$-binding})$ by Theorem~\ref{thm:325}, but
$\pi\not\in\BM(\text{hamiltonian})$, since~$\pi$ fails to satisfy
Theorem~\ref{thm:11} for $i=\bigl\lfloor\tfrac12n\bigr\rfloor-1$.

\medskip
\noindent
\textbf{4)}\quad The following was proved in~\cite{BKSWY1}.

\begin{theorem}[(Bauer et al. \cite{BKSWY1})]\label{thm:53}\mbox{}\\*
  Let~$G$ be a graph with $\bind(G)\ge2$. Then
  \begin{equation*}\label{eq:52}
    \tau(G)\ge\begin{cases}
      3/2,&\text{if $\bind(G)=2$};\\
      2,&\text{if $\bind(G)=9/4$, or $\bind(G)=2+1/(2m-1)$, for some
        $m\ge2$};\\
      2+1/m,&\text{if $\bind(G)=2+2/(2m-1)$, for some $m\ge2$};\\
      \bind(G),&\text{otherwise}.
    \end{cases}
  \end{equation*}
  Moreover, these bounds are best possible for every value of
  $\bind(G)\ge2$.
\end{theorem}

\noindent
Thus, the structural implication
$\text{$b$-binding}\Rightarrow\text{$b$-tough}$ fails to hold for
infinitely many $b\ge2$. But this implication is true in a best monotone
sense for all $b\ge2$~\cite{BKSWY2}.

\begin{theorem}[(Bauer et al. \cite{BKSWY2})]\label{thm:54}\mbox{}\\*
  If $b\ge2$, then
  $\BM(\text{$b$-binding})\Rightarrow\BM(\text{$b$-tough})$.
\end{theorem}

\noindent
The hypothesis $b\ge2$ in Theorem~\ref{thm:54} is best possible: If $m\ge2$
and $\pi=(2m-3)^{m-2}(2m-2)^2\linebreak[1](3m-4)^{2m-3}$, then taking
$b=2-1/m$, we have $\pi\in\BM(\text{$b$-binding})$ by
Theorem~\ref{thm:325}, but $\pi\not\in\BM(\text{$b$-tough})$, since~$\pi$
fails to satisfy Theorem~\ref{thm:331} for $i=2m-3$.


\begin{thebibliography}{99}

\bibitem{AS92}
N. Alon and J.H. Spencer,
\newblock \emph{The Probabilistic Method}.
\newblock John Wiley and Sons, New York, 1992.

\bibitem{A71}
I. Anderson,
\newblock Perfect matchings of a graph.
\newblock \emph{J.~Combin.\ Theory Ser.~B}~10 (1971), 183--186.

\bibitem{BBHKS12}
D. Bauer, H.J. Broersma, J. van den Heuvel, N. Kahl, and E. Schmeichel,
\newblock Degree sequences and the existence of $k$-factors.
\newblock \emph{Graphs Combin.}~28 (2012), 149--166.

\bibitem{BBHKS13}
D. Bauer, H.J. Broersma, J. van den Heuvel, N. Kahl, and E. Schmeichel,
\newblock Toughness and vertex degrees.
\newblock \emph{J.~Graph Theory}~72 (2013), 209--219.

\bibitem{BBV00}
D. Bauer, H.J. Broersma, and H.J. Veldman,
\newblock Not every 2-tough graph is hamiltonian.
\newblock \emph{Discrete Appl.\ Math.}~99 (2000), 317--321.

\bibitem{BHKS09}
D. Bauer, S.L. Hakimi, N. Kahl, and E. Schmeichel,
\newblock Sufficient degree conditions for $k$-edge-connectedness of a
graph.
\newblock \emph{Networks}~54 (2009), 95--98.

\bibitem{BHS90}
D. Bauer, S.L. Hakimi, and E. Schmeichel,
\newblock Recognizing tough graphs is NP-hard.
\newblock \emph{Discrete Appl.\ Math.}~28 (1990), 191--195.

\bibitem{BKSWY2}
D. Bauer, N. Kahl, E. Schmeichel, D.R. Woodall, and M. Yatauro,
\newblock Improving theorems in a best monotone sense.
\newblock \emph{Congr.\ Numer.}~216 (2013), 87--95.

\bibitem{BKSWY1}
D. Bauer, N. Kahl, E. Schmeichel, D.R. Woodall, and M. Yatauro,
\newblock Toughness and binding number.
\newblock \emph{Discrete Appl.\ Math.}~165 (2014), 60--68.

\bibitem{BKSY11}
D. Bauer, N. Kahl, E. Schmeichel, and M. Yatauro,
\newblock Best monotone degree conditions for binding number.
\newblock \emph{Discrete Math.}~311 (2011), 2037--2043.

\bibitem{BMS89}
D. Bauer, A. Morgana, and E.F. Schmeichel,
\newblock A simple proof of a theorem of Jung.
\newblock \emph{Discrete Math.}~79 (1989/90), 147--152.

\bibitem{HNBMC1T2F}
D. Bauer, A. Nevo, and E. Schmeichel,
\newblock Best monotone condition for 1-tough $\Rightarrow$ 2-factor (in
preparation).

\bibitem{HNBMC2F1T}
D. Bauer, A. Nevo, and E. Schmeichel,
\newblock Best monotone condition for 2-factor $\Rightarrow$ 1-tough (in
preparation).

\bibitem{BNS}
D. Bauer, A. Nevo, and E. Schmeichel,
\newblock Note on binding number, vertex degrees, and 1-factors (in
preparation).

\bibitem{HNVAVD}
D. Bauer, A. Nevo, and E. Schmeichel,
\newblock Vertex arboricity and vertex degrees (in preparation).

\bibitem{BNSWY14}
D. Bauer, A. Nevo, E. Schmeichel, D.R. Woodall, and M. Yatauro,
\newblock Best monotone degree conditions for binding number and cycle
structure.
\newblock \emph{Discrete Appl.\ Math.}\ (2014), available online at
\url{dx.doi.org/10.1016/j.dam.2013.12.014}.

\bibitem{BS90}
D. Bauer and E. Schmeichel,
\newblock Hamiltonian degree conditions which imply a graph is pancyclic.
\newblock \emph{J.~Combin.\ Theory Ser.~B}~48 (1990), 111--116.

\bibitem{BS12}
D. Bauer and E. Schmeichel,
\newblock Binding number, minimum degree, and cycle structure in graphs.
\newblock \emph{J.~Graph Theory}~71 (2012), 219--228.

\bibitem{B73}
C. Berge,
\newblock \emph{Graphs and Hypergraphs}.
\newblock North-Holland, Amsterdam, 1973.

\bibitem{B74}
F. Boesch,
\newblock The strongest monotone degree condition for $n$-connectedness of
a graph.
\newblock \emph{J.~Combin.\ Theory Ser.~B}~16 (1974), 162--165.

\bibitem{B69}
J.A. Bondy,
\newblock Properties of graphs with constraints on degrees.
\newblock \emph{Studia Sci.\ Math.\ Hungar.}~4 (1969), 473--475.

\bibitem{BC76}
J.A. Bondy and V. Chv\'{a}tal,
\newblock A method in graph theory.
\newblock \emph{Discrete Math.}~15 (1976), 111--135.

\bibitem{B41}
R.L. Brooks,
\newblock On colouring the nodes of a network.
\newblock \emph{Proc.\ Cambridge Philos.\ Soc.}~37 (1941), 194--197.

\bibitem{C79}
Y. Caro,
\newblock New results on the independence number.
\newblock Technical Report 05-79, Tel-Aviv University, 1979.

\bibitem{CH68}
G. Chartrand and F. Harary,
\newblock Graphs with prescribed connectivities.
\newblock In: \emph{Theory of Graphs}, Academic Press, New York, 1968,
61--63.

\bibitem{CKK68}
G. Chartrand, S.F. Kapoor, and H.V. Kronk,
\newblock A sufficient condition for $n$-connectedness of graphs.
\newblock \emph{Mathematika}~15 (1968), 51--52.

\bibitem{CK69}
G. Chartrand and H.V. Kronk,
\newblock The point arboricity of planar graphs.
\newblock \emph{J.~London Math.\ Soc.}~44 (1969), 612--616.

\bibitem{CLZ11}
G. Chartrand, L. Lesniak, and P. Zhang,
\newblock \emph{Graphs and Digraphs} (5th edition).
\newblock CRC Press, Boca Raton, FL, 2011.

\bibitem{C95}
C. Chen,
\newblock Binding number and toughness for matching extension.
\newblock \emph{Discrete Math.}~146 (1995), 303--306.

\bibitem{C72}
V. Chv\'{a}tal,
\newblock On Hamilton's ideals.
\newblock \emph{J.~Combin.\ Theory Ser.~B}~12 (1972), 163--168.

\bibitem{C73}
V. Chv\'{a}tal,
\newblock Tough graphs and hamiltonian circuits.
\newblock \emph{Discrete Math.}~5 (1973), 215--228.

\bibitem{C90}
W.H. Cunningham,
\newblock Computing the binding number of a graph.
\newblock \emph{Discrete Appl.\ Math.}~27 (1990), 283--285.

\bibitem{D52}
G.A. Dirac,
\newblock Some theorems on abstract graphs,.
\newblock \emph{Proc.\ London Math.\ Soc.~(3)}~2 (1952), 69--81.

\bibitem{EE89}
Y. Egawa and H. Enomoto,
\newblock Sufficient conditions for the existence of $k$-factors.
\newblock In: \emph{Recent Studies in Graph Theory}, Vishwa, Gulbarga,
1989, 96--105.

\bibitem{G68}
T. Gallai,
\newblock On directed paths and circuits.
\newblock In: \emph{Theory of Graphs}, Academic Press, New York, 1968,
115--118.

\bibitem{HS89}
S.L. Hakimi and E.F. Schmeichel,
\newblock A note on the vertex arboricity of a graph.
\newblock \emph{SIAM J.~Discrete Math.}~2 (1989), 64--67.

\bibitem{HR18}
G.H. Hardy and S. Ramanujan,
\newblock Asymptotic formulae in combinatory analysis.
\newblock \emph{Proc.\ London Math.\ Soc.}~17 (1918), 75--115.

\bibitem{H95}
C.T. Ho\`{a}ng,
\newblock Hamiltonian degree conditions for tough graphs.
\newblock \emph{Discrete Math.}~142 (1995), 121--139.

\bibitem{J78}
H.A. Jung,
\newblock On maximal circuits in finite graphs.
\newblock \emph{Ann.\ Discrete Math.}~3 (1978), 129--144.

\bibitem{KM81}
V.G. Kane and S.P. Mohanty,
\newblock Binding numbers, cycles and complete graphs.
\newblock In: \emph{Combinatorics and Graph Theory}, Lecture Notes in
Math., 885, Springer, Berlin, 1981, 290--296.

\bibitem{KW89}
P. Katerinis and D.R. Woodall,
\newblock Binding numbers of graphs and the existence of $k$-factors.
\newblock \emph{Quart.\ J.~Math.\ Oxford Ser.~(2)}~38 (1987), 221--228.

\bibitem{K07}
M. Kriesell,
\newblock Degree sequences and edge connectivity.
\newblock Preprint, available online at
\url{www.math.uni-hamburg.de/research/papers/hbm/hbm2007282.pdf}, 2007.

\bibitem{K69}
H.V. Kronk,
\newblock A note on $k$-path hamiltonian graphs.
\newblock \emph{J.~Combin.\ Theory}~7 (1969), 104--106.

\bibitem{V72}
M. Las Vergnas,
\newblock \emph{Probl{\`e}mes de Couplages et Probl{\`e}mes Hamiltoniens en
  Th{\'e}orie des Graphes.}
\newblock PhD Thesis, Universit{\'e} Paris VI -- Pierre et Marie Curie,
1972.

\bibitem{L76}
L. Lesniak,
\newblock On $n$-hamiltonian graphs.
\newblock \emph{Discrete Math.}~14 (1976), 165--169.

\bibitem{LP09}
J. Lyle and W. Goddard,
\newblock The binding number of a graph and its cliques.
\newblock \emph{Discrete Appl.\ Math.}~157 (2009), 3336--3340.

\bibitem{M91}
O. Murphy,
\newblock Lower bounds on the stability number of graphs computed in terms
of degrees.
\newblock \emph{Discrete Math.}~90 (1991), 207--211.

\bibitem{N71}
C.St.J.A. Nash-Williams,
\newblock Hamiltonian arcs and circuits.
\newblock In: \emph{Recent Trends in Graph Theory}, Lecture Notes in Math.,
186, Springer, Berlin, 1971, 197--210.

\bibitem{P62}
L. P\'osa,
\newblock A theorem concerning Hamilton lines.
\newblock \emph{Magyar Tud.\ Akad.\ Mat.\ Kutat\'o Int.\ K\"ozl.}~7 (1962),
225--226.

\bibitem{R79}
A.R. Rao,
\newblock The clique number of a graph with a given degree sequence.
\newblock In: \emph{Proc. Sympos. on Graph Theory}, ISI Lecture Notes
Series~4, 1979, 251--267.

\bibitem{RW00}
A.M. Robertshaw and D.R. Woodall,
\newblock Triangles and neighbourhoods of independent sets in graphs.
\newblock \emph{J.~Combin.\ Theory Ser.~B}~80 (2000), 122--129.

\bibitem{S87}
R. Shi,
\newblock The binding number of a graph and its pancyclism.
\newblock \emph{Acta Math.\ Appl.\ Sinica (English Series)}~3 (1987),
257--269.

\bibitem{SW68}
G. Szekeres and H.S. Wilf,
\newblock An inequality for the chromatic number of a graph.
\newblock \emph{J.~Combin.\ Theory}~4 (1968), 1--3.

\bibitem{T54}
W.T. Tutte,
\newblock A short proof of the factor theorem for finite graphs.
\newblock \emph{Canad.\ J.~Math.}~6 (1954), 347--352.

\bibitem{W81}
V.K. Wei,
\newblock A lower bound on the stability number of a simple graph.
\newblock Technical Memorandum TM~81-11217-9, Bell Laboratories, 1981.

\bibitem{WP67}
D.J.A. Welsh and M.B. Powell,
\newblock An upper bound for the chromatic number of a graph and its
application to timetabling problems.
\newblock \emph{Comput.\ J.}~10 (1967), 85--86.

\bibitem{W01}
D. West,
\newblock \emph{Introduction to Graph Theory} (2nd edition).
\newblock Prentice Hall, Upper Saddle River, NJ, 2001.

\bibitem{W67}
H.S. Wilf,
\newblock The eigenvalues of a graph and its chromatic number.
\newblock \emph{J.~London Math.\ Soc.}~42 (1967), 330--332.

\bibitem{W73}
D.R. Woodall,
\newblock The binding number of a graph and its Anderson number.
\newblock \emph{J.~Combin.\ Theory Ser.~B}~15 (1973), 225--255.

\bibitem{W78}
D.R. Woodall,
\newblock A sufficient condition for hamiltonian circuits.
\newblock \emph{J.~Combin.\ Theory Ser.~B}~25 (1978), 184--186.

\bibitem{W90}
D.R. Woodall,
\newblock $k$-factors and neighbourhoods of independent sets in graphs.
\newblock \emph{J.~London Math.\ Soc.~(2)}~41 (1990), 385--392.

\bibitem{HGTA}
J.-H. Yin and J.-Y. Guo,
\newblock Forcibly $k$-edge-connected graphic sequences
\newblock (to appear).

\end{thebibliography}
\end{document}